\documentclass[11pt, reqno]{amsart}
\usepackage{amsmath, amssymb, amsthm, graphicx,marvosym}
\usepackage{cancel}
\usepackage[nobysame]{amsrefs}[11pts]

\usepackage{color}

\usepackage{tikz}
\usepackage{amsmath}
\usepackage{amssymb}
\usepackage{amsthm}
\usepackage{enumerate}
\newtheorem{thm}{Theorem}
\newtheorem{lemma}{Lemma}[section]

\newtheorem{prop}[lemma]{Proposition}
\newtheorem{definition}[lemma]{Definition}
\newtheorem{remark}[lemma]{Remark}
\newcommand \nc{\newcommand}
\newcommand{\ben}{\begin{eqnarray}}
\newcommand{\een}{\end{eqnarray}}
\newcommand{\beno}{\begin{eqnarray*}}
\newcommand{\eeno}{\end{eqnarray*}}

\makeatletter \@addtoreset{equation}{section} \makeatother

\nc{\ba}{\begin{array}}\nc{\ea}{\end{array}}
\nc{\be}{\begin{eqnarray}}\nc{\ee}{\end{eqnarray}}
\nc{\beq}{\begin{equation}}\nc{\eeq}{\end{equation}}
\nc{\bex}{\begin{eqnarray*}}\nc{\eex}{\end{eqnarray*}}
\nc{\btm}{\begin{theorem}} \nc{\etm}{\end{theorem}}
\nc{\blm}{\begin{lemma}} \nc{\elm}{\end{lemma}}
\nc{\va}{\varphi}
\nc{\ve}{\varepsilon}

\def\di{\mbox{div\,}}
\def\curl{\mbox{curl\,}}
\def\ds{\displaystyle}

\newcommand{\e}{\varepsilon}

 \newcommand{\n}{{\bf n}}
\newcommand{\bu}{{\bf u}}

\newcommand{\R}{\mathbb {R}}

\def\com#1{\quad{\textrm{#1}}\quad}

\def\nn{\nonumber}

\def\({\left(\begin{array}{cccccc}}
\def\){\end{array}\right)}

\def\bes{\begin{eqnarray}}
\def\ees{\end{eqnarray}}

\begin{document}
\title[Wave model for liquid-crystals]{Poiseuille flow of nematic liquid crystals via the full Ericksen-Leslie model}

\author{Geng Chen}
\address[G. Chen]{Department of Mathematics, University of Kansas, Lawrence, KS 66045, U.S.A.}
\email{\tt gengchen@ku.edu}

\author{Tao Huang}
\address[T. Huang] {Department of 
Mathematics, Wayne State University, Detroit, MI, 48202, U.S.A.,
 and NYU-ECNU Institute of Mathematical Sciences, Shanghai New York University, Shanghai, 200062, P.R.C.}
\email{\tt taohuang@wayne.edu}

\author{Weishi Liu}
\address[W. Liu]{Department of Mathematics,
University of Kansas, Lawrence, KS 66045, U.S.A.}
\email{\tt wsliu@ku.edu}
\date{\today}

\begin{abstract} 
  In this paper, we study the Cauchy problem of the Poiseuille flow of the full Ericksen-Leslie model for nematic liquid crystals.
  The model is a coupled system of a parabolic equation for the velocity 
  and a quasilinear wave equation for the director.
 For a particular choice of several physical parameter values, we construct solutions with smooth initial data and finite energy that produce, in finite time, cusp singularities -- blowups of gradients. The formation of cusp singularity is due to local interactions of wave-like characteristics of solutions, which is different from the mechanism of finite time singularity formations for the parabolic Ericksen-Leslie system. 
The finite time singularity formation for the physical model  might raise some concerns for purposes of applications. 
This  is, however,  resolved satisfactorily; more precisely, we  are able to establish 
   the global existence of   weak solutions that are  H\"older continuous and have bounded energy. 
   One major contribution of this paper is our identification of the effect of the flux density of the velocity  on the director  and the  reveal of a singularity cancellation --   the flux density remains bounded while its   two components approach infinity at formations of  cusp singularities.

\end{abstract}
\maketitle



\section{Introduction}
In this paper, we consider singularity formation and global existence of H\"older continuous weak solution for the Cauchy problem  
\begin{align}\label{simeqn0}
\begin{split}
 u_t=&(u_x+\theta_t)_x,\\
\theta_{tt}+2\theta_t=&c(\theta)(c(\theta)\theta_{x})_x
 -u_x.
\end{split}
\end{align}
with initial data 
\beq\label{initial}
u(x,0)=u_0(x)\in H^1(\R),\quad \theta(x,0)=\theta_0(x)\in H^1(\R),
\quad \theta_t(x,0)=\theta_1(x)\in L^2(\R).
\eeq
In addition, we assume that, for some $\alpha>0$, 
\beq\label{tecinitial}
u'_0(x)+\theta_1(x)\in L^\infty\cap C^\alpha(\R)\; \hbox{ and}\; \lim_{|x|\rightarrow\infty} (\theta_1, \theta'_0, u'_0)(x)=0.
\eeq

We also assume that the function $c(\cdot)$ is  $C^2$, and there exist positive constants $C_L$, $C_U$ and $C_1$ such that,
\beq\label{condc}
0<C_L\leq c(\cdot)\leq C_U<\infty, \quad |c'(\cdot)|\leq C_1.
\eeq

System (\ref{simeqn0}) is  the full Ericksen-Leslie model for Poiseuille flow of nematic liquid crystals with a particular choice of parameters.  The general model for  Poiseuille flow of nematics  and the choice of parameters that leads to system (\ref{simeqn0}) will be discussed in Section \ref{PoiSys}. 

For system (\ref{simeqn0}),  we will show by a class of examples that (one-sided) cusp singularity can be formed in finite time and, taking this into consideration, we are still able to  
 establish the global existence of weak solutions with a bounded energy.
Although we do not work on the model for Poiseuille flow of nematics with general parameters in this paper, we believe that the similar results hold true in general. 

We will next recall the Ericksen-Leslie model followed by a discussion of some relevant results to this work. Experts in this field can skip Section \ref{sec2.1} and jump to Section \ref{RelResult}.

\subsection{Ericksen-Leslie  model for nematic liquid crystals\label{sec2.1}}

Liquid crystals are intermediate phases between solid and isotropic fluid. Liquid crystal materials  have a degree of crystal structures but also  exhibit many hydrodynamic features  so they are capable to flow.  These multi-facet properties are very important to present applications of display and many yet to come.   Nematic liquid crystals are composed of rod-like molecules characterized by average alignment of  the long axes of neighboring molecules, which have simplest structures  among liquid crystals and have been widely studied analytically and experimentally that lead to fruitful applications (\cite{DeGP, Cha, Eri76, Les}). 
The modeling and analysis of the nematic liquid crystals have attracted  a lot of interests of mathematicians for several decades. 

 If the orientation order parameters of nematic materials are treated as a unit vector ${\n}\in \mathbb S^2$, the director, then the Oseen-Frank energy density determines the macrostructure of the crystal structure (\cite{oseen33,frank58})
\begin{align}\label{OFE}\begin{split}
2W(\n,\nabla \n)=&K_1(\di \n)^2+K_2(\n\cdot\curl \n)^2+K_3|\n\times\curl \n|^2\\
&+(K_2+K_4)[\mbox{tr}(\nabla \n)^2-(\di \n)^2 ]
\end{split}
\end{align}
where $K_j$, $j=1,2,3$, are the positive constants representing splay, twist, and bend effects respectively, with
$K_2\geq |K_4|$, $2K_1\geq K_2+K_4$. (One often takes $K_2+K_4=0$.) 
The equilibrium theory of nematics is the variational problem of the total Oseen-Frank energy over the domain $\Omega\subset \mathbb R^3$ occupied by the material. The theory  has been developed successfully and gives a wide range of interesting properties \cite{HardtK87, linlius01, linwangs14}. 
 
The hydrodynamic property of nematics is macroscopically characterized by the velocity field $\bu$. Any distortion of the director $\n$ causes the flow and, likewise,  any flow affects the alignment $\n$. These influences are determined by the 
the kinematic transport tensor $g$ and the viscous stress tensor $\sigma$ given below. Let
$$
D= \frac12(\nabla \bu+\nabla^{T} \bu),\quad \omega= \frac12(\nabla \bu-\nabla^{T}\bu),\quad N=\dot \n-\omega \n, 
$$
represent the rate of strain tensor, skew-symmetric part of the strain rate, and the
rigid rotation part of director changing rate by fluid vorticity, respectively. The kinematic transport ${\bf g}$ is given by
\begin{align}\label{g}
{\bf g}=\gamma_1 N +\gamma_2D\n 
\end{align}
which represents the effect of the macroscopic flow field on the microscopic structure. The material coefficients $\gamma_1$ and $\gamma_2$ reflect the molecular shape and the slippery part between fluid and particles. The first term of ${\bf g}$ represents the rigid rotation of molecules, while the second term stands for the stretching of molecules by the flow.
The viscous (Leslie) stress tensor $\sigma$ has the following form
\begin{align}\label{sigma}\begin{split}
\sigma=& \alpha_1 (\n^TD\n)\n\otimes \n +\alpha_2N\otimes \n+ \alpha_3 \n\otimes N\\
&   + \alpha_4D + \alpha_5(D\n)\otimes \n+\alpha_6\n\otimes (D\n),
\end{split}
\end{align}
where $\mathbf{a}\otimes \mathbf{b}=\mathbf{a}\, \mathbf{b}^T$ for column vectors $\mathbf{a}$ and $\mathbf{b}$ in $\mathbb{R}^n$. These coefficients $\alpha_j$ $(1 \leq j \leq 6)$, depending on material and temperature, are called Leslie coefficients. The following relations are assumed in the literature.
\begin{align}\label{a2g}
\gamma_1 =\alpha_3-\alpha_2,\quad \gamma_2 =\alpha_6 -\alpha_5,\quad \alpha_2+ \alpha_3 =\alpha_6-\alpha_5. 
\end{align}
The first two relations are compatibility conditions, while the third relation is called Parodi's relation, derived from Onsager reciprocal relations expressing the equality of certain relations between flows and forces in thermodynamic systems out of equilibrium (cf. \cite{Parodi70}). 
They also satisfy the following empirical relations (p.13, \cite{Les}) 
\begin{align}\label{alphas}
&\alpha_4>0,\quad 2\alpha_1+3\alpha_4+2\alpha_5+2\alpha_6>0,\quad \gamma_1=\alpha_3-\alpha_2>0,\\
&  2\alpha_4+\alpha_5+\alpha_6>0,\quad 4\gamma_1(2\alpha_4+\alpha_5+\alpha_6)>(\alpha_2+\alpha_3+\gamma_2)^2\notag.
\end{align}
Note that the 4th relation  is implied by the 3rd together with the last relation and the last can be rewritten as
$\gamma_1(2\alpha_4+\alpha_5+\alpha_6)> \gamma_2^2$.

The dynamic theory of nematics   was first proposed by Ericksen \cite{ericksen62} and Leslie \cite{leslie68}  in the 1960's.   Using the convention to denote $\dot{f}=f_t+\bu\cdot \nabla f$ the material derivative,
the full Ericksen-Leslie system is given as follows (see, e.g. \cite{Les, lin89})
\begin{equation}\label{wlce}
\begin{cases}
\rho\dot \bu+\nabla P=\nabla\cdot\sigma-\nabla\cdot\left(\frac{\partial W}{\partial\nabla \n}\otimes\nabla \n\right),
\\
\nabla\cdot \bu=0,\\ 
\nu{\ddot \n}=\lambda\n-\frac{\partial W}{\partial  \n}-{\bf g}+\nabla\cdot\left(\frac{\partial W}{\partial\nabla \n}\right), \\
|\n|=1.\\
\end{cases}
\end{equation}
In  (\ref{wlce}),  $P$ is the pressure, $\lambda$ is the Lagrangian multiplier of the constraint $|\n|=1$, $\rho$ is the density, $\nu$ is the inertial coefficient of the director $\n$, $W$ is the Oseen-Frank energy in (\ref{OFE}),  ${\bf g}$  and $\sigma$ are the kinematic transport and the viscous stress tensor, respectively, given in (\ref{g}) and (\ref{sigma}).

\subsection{Results relevant to present work.} \label{RelResult}
The full Ericksen-Leslie system \eqref{wlce}  is a coupled system of   forced Navier-Stokes equations and the wave map equations. Basic concerns about existence, uniqueness and regularity of solutions are not completely understood. 
{In general, global regular solutions are not expected; in fact, in several cases, singularity is shown to formulate in finite time for smooth initial data. Therefore, {\em singularity formation} and {\em global existence of weak solutions} are often treated in pair for dynamical models of liquid crystals from mathematical analysis viewpoint. This is the case of this work. }

\subsubsection{On the variational wave equation for director field}\label{varWave} When the fluid field $\bu$ is neglected, the Ericksen-Leslie system \eqref{wlce} is replaced by  a quasilinear wave system only on the director field $\n$. (It is known that the neglect of $\bu$ is not physically consistent since a change of $\n$ in time would drive a change of $\bu$ in time.) In one spatial dimension $x\in\mathbb R$ and for director $\n(x,t)=(\cos(\theta(x,t)), 0,\sin(\theta(x,t))$ restricted to a unit circle,  the quasilinear wave system --  the second equation in (\ref{simeqn0}) without $u_x$ and the damping $2\theta_t$ -- is often called {\em the variational wave equation} and was intensively studied in the last two decays (see for example \cites{GHZ}).  

Solutions of the variational wave equation with smooth initial data could in general produce cusp singularities due to local interactions of waves; more precisely, there may be finite time blowup in their gradients while the solutions themselves are still H\"older continuous (\cites{GHZ,CHL,CZ12}). 
On the other hand, the existence of global energy conservative solutions after the singularity formation was established in \cite{BZ}. Later this result was extended to more general initial data in \cite{HR}, the case with damping in \cite{CZ12} and the variational wave system with $\n\in{\mathbb S}^2$ in  \cite{CZZ12,ZZ10,ZZ11}. Especially, in \cite{CZ12}, the authors showed that   behaviors of large solutions of the variational wave systems with and without damping are similar.
The global well-posedness  of H\"older continuous conservative solutions was established for the variational wave system, including:  uniqueness
\cites{BCZ,CCD}, Lipschitz continuous dependence on some optimal transport metric \cite{BC2015}, and generic regularity \cites{BHY,BC}. 
The existence of the dissipative solution was studied in \cites{BH,ZZ03}.

The  singularity formation of the variational wave equation is due to local interactions of waves. This mechanism is different from that for the parabolic Ericksen-Leslie models which will be discussed in the next part \S \ref{paraEL}.

The singularity formation of the present work on system (\ref{simeqn0}) is inspired by and  directly related to those for the variational wave equations discussed above.  A major difference is the coupling term $u_x$ on $\theta$ in the second equation of system (\ref{simeqn0}). It turns out $u_x$ blows up when singularity forms, which makes it hard to track its effect on the singularity  of $\theta$ from the variational wave equation.
We are able to control the  effect of $u_x$ by controlling that of the quantity  $J(x,t):=u_x+\theta_t$, and show that the singularity formation for the coupled system (\ref{simeqn0}) has more or less the same mechanism as that for the variational wave equation.  

Note that, from the first equation of system (\ref{simeqn0}), the quantity  $J(x,t):=u_x+\theta_t$ is the flux density of the velocity $u$.
The flux density $J(x,t)$ of the velocity further plays a crucial role in establishing the existence of global weak solutions.  For the global existence result, we 
adapt  the framework in \cite{BZ} of using the semilinear system on characteristic coordinates for the variational wave equation. For our problem (\ref{simeqn0}), however, in the heat equation, the solution flow does not propagate along characteristic directions.  One has to overcome the difficulty caused by the coupling of ``mismatching"  behaviors.   A key ingredient for  extending the framework in \cite{BZ} to the coupled system (\ref{simeqn0}) is a careful  treatment  of the flux density $J(x,t)$ of the velocity. In fact $J(x,t)$ will be shown to be bounded (see Lemma \ref{lemma3.1}), although $u_x$ and $\theta_t$ both may blowup in finite time.

\subsubsection{On the parabolic Ericksen-Leslie system}\label{paraEL} When $\nu=0$, the  Ericksen-Leslie system \eqref{wlce} becomes a parabolic system (also called Ericksen-Leslie system in literature).  For the parabolic Ericksen-Leslie system in dimension two, the existence and uniqueness of global solution have been studied in \cite{WZZ13, huanglinwang14, hongxin12, LTX16, wangwang14}.   In dimension three, under some simplified assumptions, the authors of \cite{WZZ13} established global existence of solutions for small initial data and provided a characterization of the maximal existence time for general initial data.  

In \cite{lin89}, Lin proposed a simplified system, by neglecting the Leslie stress and taking $W(\n,\nabla \n)$ to be the Dirichlet energy density
\begin{equation}\label{lce}
\begin{cases}
\bu_t+\bu\cdot\nabla \bu-\Delta \bu+\nabla P
=-\nabla\cdot\left(\nabla \n \odot\nabla \n-\frac{1}{2}|\nabla \n|^2\mathbb{I}_3\right)
 \\
\nabla\cdot \bu=0  \\
\n_t+\bu\cdot\nabla \n=\Delta \n+|\nabla \n|^2\n.
\end{cases}
\end{equation}
For system (\ref{lce}) in dimension two, it was shown in \cite{linlinwang10, linwang10} that there is a unique Leray-Hopf type global weak solution. This weak solution may have at most finitely many singular times, at which $|\bu|+|\nabla \n|\rightarrow \infty$. Very recently, examples of finite time singularities of such weak solution have been constructed in \cite{llwwz19} by a new inner-outer gluing method. More precisely, given any $k$ points in the domain of dimension two, the local smooth solution blows up exactly at those $k$ points at finite time. In dimension three, existence of global weak solutions has been shown in \cite{linwang16} under the assumption $\n_0(x)\in \mathbb S^2_+$ with the help of some new compactness arguments.  In \cite{HLLW16}, for system (\ref{lce})  over a bounded domain,   two examples of finite time singularity have been constructed. {The formations of these singularities are related to some  {\em global or non-trivial topological} conditions on the initial data (over bounded domains); in particular, the mechanisms are different from that for variational wave equation discussed in the previous part \S  \ref{varWave} and our system \eqref{simeqn0} (see Section \ref{sec-results} for more details).}   It is not clear how the singularity will behave after its formation, which is presumably one of the main difficulties in establishing a global existence result. 

Although   system (\ref{lce})  misses specifics of many physical parameters, the simplification allows initial success in analyzing the general dynamical behavior of such a system that further drives a great deal treatments of the parabolic Ericksen-Leslie system.  For a more complete review, please see the survey paper \cite{linwangs14} and the references therein.

\subsubsection{On the full Ericksen-Leslie system (\ref{wlce})}
 The full Ericksen-Leslie system \eqref{wlce} itself is poorly understood. It seems to the authors that the only result available  for the global wellposedness  is in \cite{jiangluo17} where local existence and uniqueness for initial data with finite energy and global existence and uniqueness of classical solutions with small initial data were established.

%

\subsection{Main results of this work}\label{sec-results}
An   interesting and important question is the existence and behaviors of global solutions for the full Ericksen-Leslie system \eqref{wlce} with $\nu>0$.
In this paper, we give an example of  singularity formation and establish the global existence of weak solutions for the special Poiseuille flow (\ref{simeqn0}).
\smallskip

\paragraph{\underline{\bf Finite time singularity formation}}

Inspired by \cite{GHZ, CZ12,CHL}, we can construct some special smooth initial data for which the solution will produce singularity of gradient blow-up in finite time. 
To this end, we introduce a $C^1(\mathbb R)$ function $\phi$ satisfying the following properties
\beq\label{id_sing1}
\phi(0)=0\;\mbox{ and }\; \phi(a)=0\;\mbox{ for }\; a\not\in(-1,1),
\eeq
\beq\label{id_sing2}
-  \phi'(0)>\max\Big\{\frac{16C_U}{c'(\theta^*) C_L},\frac{2}{C_L}\Big\}\,,\quad |\phi'(x)|\leq C_2
\eeq
where $C_L$ and $C_U$ are defined in \eqref{condc}, $C_2$ is a positive constant and
\beq\label{id_sing3}
\int_{-1}^1(\phi')^2(a)\,da<k_0,
\eeq
for some constant $k_0$. 

\begin{thm}\label{sing}
Consider the Cauchy problem of \eqref{simeqn0}-\eqref{condc} with the following $C^1$ initial data 
\beq\label{id_sing0}
\theta_0(x)=\theta^*+ \e\, \phi(\frac{x}{\e}),\quad \theta_1(x)=\left( - c(\theta_0(x))+\e \right)\, \theta'_0(x)
\eeq
\beq\label{id_sing00}
u_0(x)=\left\{
\begin{array}{lll}\vspace{2mm}
&0,\quad& x\in (-\infty,-\ve),\\\vspace{2mm}
&\displaystyle\int_{-\ve}^x\, c(\theta_0(a))\, \theta_0'(a)\,da,\quad &x\in [-\ve,\ve]\\
&\Phi(x),\quad& x\in (\ve,\ve+2),\\
&0,\quad& x\in (\ve+2,\infty),
\end{array}
\right.
\eeq 
where $\theta^*$ is a constant satisfying
$c'(\theta^*)>0$, $\phi(x)$ is the function satisfying (\ref{id_sing1})-(\ref{id_sing3}), and $\Phi(x)$ is $C^1$ and satisfies
\beq\label{asmPhi1}
\Phi(\ve)=\int_{-\ve}^{\ve}\, c(\theta_0(a))\, \theta_0'(a)\,da, \quad \Phi'(\ve)=c(\theta_0(\ve))\theta_0'(\ve)=0, 
\eeq
\beq\label{asmPhi2}
\Phi(\ve+2)=\Phi'(\ve+2)=0,\quad |\Phi'(x)|\leq 6C_UC_2 \ve \mbox{ for any }x\in (\ve,\ve+2)
\eeq

Then,  one can choose $\e>0$ sufficiently small, 
such that the solution $(u(x,t),\theta(x,t))$ is $C^1$ only for $t<t_*$ with some $t_*<1$ and forms singularity as $t\rightarrow t_*^-$; more precisely, at some $x_*$,  
$$\theta_t(x,t)\to \infty,\quad \theta_x(x,t)\to -\infty,\quad u_x(x,t)\to -\infty$$ 
as $(x,t)\to (x_*,t_*^-)$.
\end{thm}

We comment that the requirement on $\Phi$ in (\ref{asmPhi1}) and (\ref{asmPhi2}) is consistent; in fact one can construct a function $\Phi$ with all properties and with the factor $6$ in (\ref{asmPhi2}) being replaced by any number bigger than $4$.


Together with the energy decay for smooth solutions, we know that the singularity formed in finite time is  a cusp   (generically  one-sided-cusp) singularity, i.e. derivatives $|\theta_x|$ and $|\theta_t|$ are infinity (see \cite{GHZ}), but the $L^2$ norms of $|\theta_x|$ and $|\theta_t|$ are finite by Proposition \ref{E4smooth}, which gives H\"older continuity of $\theta$. The estimate on $J=u_x+\theta_t$ in Lemma \ref{lemma3.1} and the relation $u_x=J-\theta_t$ show that $u(\cdot,t)$ is also H\"older continuous for almost all $t$. See  Remark \ref{rem} for more details.

 As mentioned in Section \ref{RelResult}, the two examples of finite time singularity formation constructed in \cite{HLLW16, llwwz19} for the parabolic system over bounded regions are directly related to or caused by some non-trivial global/topological conditions.  While as the singularity  claimed  in Theorem \ref{sing} is formed   in essentially the  same mechanism as that in \cite{CZ12, GHZ} -- it is created {\em locally} due to  interactions of local waves that are of finite speed.    A typical point singularity of direction field $\n$ of three dimensional parabolic system is in the form of $x/|x|$, which is not continuous at singular point. In fact, if $(\mathbf{u}, \n)$ is continuous, one may show higher regularity of the solutions to parabolic system.

  Our method of showing the formation of singularity is thus based on those in papers \cite{CHL, CZ12, GHZ} for a variational wave equation while there are several new ideas provided to cope with the new model. For example, we need understand the impact of the source term $u_x$ in $\eqref{simeqn0}_2$. 
\smallskip

\paragraph{\underline{\bf Global existence of weak solutions}} Due to the formation of singularity in Theorem \ref{sing}, 
 one cannot expect existence of global classical solutions in general.  One would like to  know how the singularity behaves and whether a certain class of weak solutions   exist  beyond the time of singularity formation. This is important particularly for models of physical problems that are expected to have ``global solutions''.
 We will show that      a weak solution defined below does exist globally and has a bounded energy.
 
\begin{definition}\label{def1}   For any given time $T<\infty$,   $(u(x,t),\theta(x,t))$  is a weak solution to the initial value problem \eqref{simeqn0}-\eqref{tecinitial} for $(x,t)\in\R\times[0,T]$ if
\begin{itemize}
\item[(i)] for any $\phi\in C^\infty_0\{\R\times [0,T]\}$,
\begin{align}\label{weak1}
\iint \left(\theta_t \phi_t-(c(\theta)\phi)_x c(\theta)\theta_x -\theta_t \phi-v_t\phi\right)\,dxdt=0
\end{align}
 with
\beq\label{sec5u}
v(x,t)=\int_{-\infty}^x u(y,t) dy,\quad\hbox{and}\quad v_x(x,t)=u(x,t)
\eeq
pointwise with 
\[{  v_t(x,t)\in L^\infty\cap L^2(\R\times[0,T])}\] 
and 
\[v_t=v_{xx}+\theta_t\]
 is satisfied in $L^2(\R\times[0,T])$ sense, and
\[
u\in L^2([0,T],H^1(\R))\cap L^\infty([0,T],H^1_{loc}(\R))\cap 
L^\infty([0,T]\times\R),
\]
and
\[
u_t\in L^2([0,T],H^{-1}(\R)).
\]
\item[(ii)]   the first and second equations for initial conditions in \eqref{initial} are satisfied pointwise, and
the third equation holds in $L^p_{loc}$ for $p\in[1,2)$.
\end{itemize}
\end{definition}  
 
\begin{thm}\label{thmplc1} Assume $c(\theta)$ satisfies (\ref{condc}) and  $\theta_0$ is absolutely continuous. Then, for any time $T<\infty$, there exists a weak solution $(u(x,t),\theta(x,t))$ in the sense of Definition \ref{def1} for $(x,t)\in\R\times[0,T]$ to the initial value problem \eqref{simeqn0}-\eqref{tecinitial}. 
Furthermore,  
\begin{itemize}
\item[(i)]  the associated energy 
\beq\label{E.00}
\mathcal{E}(t): =\frac{1}{2}\int_{\R} \left(\theta_t^2+c^2(\theta)\theta_x^2+ u^2\right)\,dx
\eeq
is well-defined   for $t\in (0,T]$ and satisfies  
\[\mathcal{E}(t)\le \mathcal{E}(0)  -\iint_{\R\times [0,t]}(v_t^2+\theta^2_t)\,dxdt;\]
\item[(ii)]  $\theta(x,t)$ is locally H\"older continuous with exponent $1/2$ in both $x$ and $t$;
\item[(iii)] $u(x,t)$ is locally H\"older continuous in $x$ with exponent $1/2$ for  a.e. $t$.
\end{itemize}
\end{thm}
 
  Note that the statement of Theorem \ref{thmplc1} involves an arbitrary but fixed time $T<\infty$. The reason is that we do not have uniqueness on weak solutions. Thus, in principle, for different $T$, one may have different weak solutions with the same initial data that make it difficult to get the conclusion for $(x,t)\in \R\times [0,\infty)$. Of course, we do not believe neither suggest the latter is the case. 

A main challenge in establishing a global existence   comes from the coupling of quasilinear wave equation $\eqref{simeqn0}_2$ and heat equation $\eqref{simeqn0}_1$. To solve the quasilinear wave equation $\eqref{simeqn0}_2$ without $2\theta_t$ and $u_x$ for general initial data, one of  few available frameworks is to use a semilinear system on some dependent variables in the energy dependent characteristic coordinates introduced in \cite{BZ}.   However, in the heat equation $\eqref{simeqn0}_1$, the solution flow does not propagate along characteristic directions, so it destroys the sharp wave front. Here   the source term $\theta_{tx}(\cdot,t)$ in $\eqref{simeqn0}_1$ has a poor regularity, only $H^{-1}$, since $\theta_t(\cdot,t)\in L^2$ for any $t$. So the solution cannot gain any regularity directly from the heat equation $\eqref{simeqn0}_1$. As mentioned in \S \ref{varWave},   a key ingredient for  extending the framework in \cite{BZ} to our coupled system is a careful  treatment  of  the flux density  $J(x,t)=u_x+\theta_t$ of the velocity.

\medskip

The remaining of the paper is organized as follows.   In Section \ref{PoiSysIdea}, we discuss the model for Poiseuille flows of nematics, specify the choice of parameters that leads to system (\ref{simeqn0}) considered in this paper, and explain main ideas for the proofs of our results.  In Section \ref{Lfun}, we give the a priori estimate on the flux density  $J(x,t)$ of the velocity for smooth solutions.  In Section \ref{SF}, we construct a singularity formation example. In Section \ref{SLS},  a semilinear system for the wave equation will be given.  In Section \ref{GER}, we prove the existence of weak solutions and the energy estimate. In Appendix A, we provide a brief derivation of (\ref{time_poi}) for the Poiseuille flows of nematics, a derivation of 
the semilinear system in the characteristic coordinates used in Section \ref{GER}, and a proof of the H\"older continuity of some functions used in  Section \ref{sec_6.2}.

\section{Poiseuille flows,  special system (\ref{simeqn0}),   ideas of our analysis}\label{PoiSysIdea} 

\subsection{System for Poiseuille flows and energy decay for smooth solutions} \label{PoiSys} 
In this paper, we are interested in Poiseuille flows of nematic liquid crystals; more precisely,   we will consider solutions of  system \eqref{wlce} of the form  (\cite{liucalderer00})
\[\bu(x,t)=(0,0,u(x,t))^T\;\mbox{ and }\;
\n(x,t)=\big(\sin\theta(x,t), 0, \cos\theta(x,t)\big)^T.\]
Then system \eqref{wlce} becomes (see Appendix \ref{Sec_A1} for a detailed derivation)
\begin{align}\label{time_poi}
\begin{split}
 \rho u_t=&a+\Big(g(\theta)u_x+h(\theta)\theta_t\Big)_x,\\
\nu\theta_{tt}+\gamma_1\theta_t=&c(\theta)(c(\theta)\theta_{x})_x
 -h(\theta)u_x,
\end{split}
\end{align}
where      the constant $a$ is the gradient of pressure along the flow direction, and 
 \begin{align}\label{fgh}\begin{split}
 g(\theta):=&\alpha_1\sin^2\theta\cos^2\theta+\frac{\alpha_5-\alpha_2}{2}\sin^2\theta+\frac{\alpha_3+\alpha_6}{2}\cos^2\theta+\frac{\alpha_4}{2},\\
 f(\theta)\,\equiv&c^2(\theta):=K_1\cos^2\theta+K_3\sin^2\theta,\\
  h(\theta):=&\alpha_3\cos^2\theta-\alpha_2\sin^2\theta=\frac{\gamma_1+\gamma_2\cos(2\theta) }{2}.
 \end{split}
 \end{align}
 The last relation comes from \eqref{a2g}. Note that  $c(\cdot)$ is smooth and satisfies \eqref{condc}. 

 For system (\ref{time_poi}),  we will take $a=0$ in the sequel.  In fact, once one finds a weak solution $(u,\theta)$ for \eqref{time_poi} without $a$, then $(\hat u=u+\frac{a}{\rho}t,\theta)$ will satisfy system \eqref{time_poi} with $a$.  For any smooth solution of the system \eqref{time_poi},  we define  the associated energy
\beq\label{E.0}
\mathcal{E}(t):=\frac{1}{2}\int_{\R} \left(\nu\theta_t^2+c^2(\theta)\theta_x^2+\rho u^2\right)\,dx.
\eeq
Let   $b(\theta)$ be the function given by  
\begin{align}\label{btheta} \begin{split}
b(\theta):=g(\theta)-\frac{1}{\gamma_1}h^2(\theta) 
=&\frac{\gamma_1(2\alpha_4+\alpha_5+\alpha_6)-\gamma_2^2}{4\gamma_1} \cos^2(2\theta)
+\frac{\alpha_4}{8}\sin^2(2\theta)\\
&+\frac{2\alpha_1+3\alpha_4+2\alpha_5+2\alpha_6}{8}\sin^2(2\theta).
\end{split}
\end{align}
Then $b(\theta)> 0$  is an immediate  consequence of (\ref{a2g}) and (\ref{alphas}). 

  \begin{prop}\label{E4smooth}
   If $(u(x,t),\theta(x,t))$ is a smooth solution  of the Poiseuille flow \eqref{time_poi} with $a=0$, then  the associated energy $\mathcal{E}(t)$ decays; more precisely,
\begin{align}\label{engineq}
\begin{split}
\frac{d}{dt}\mathcal{E}(t)= -\int_{\R}\Big(b(\theta)u_x^2+\gamma_1\Big( \theta_t+\frac{h(\theta)}{\gamma_1}u_x\Big)^2\Big)\;dx
\leq 0,
\end{split}
\end{align}
where  $b(\theta)>0 $ is given in (\ref{btheta}).
\end{prop}
\begin{proof} Recall, with $a=0$, the system \eqref{time_poi} becomes
\begin{equation}\label{wpfc}
\begin{cases}
\rho\ds u_t=\left(g(\theta)u_x+h(\theta)\theta_t\right)_x,\\
\nu\ds\theta_{tt}+\gamma_1\theta_t=c(\theta)\big(c(\theta)\theta_x\big)_x-h(\theta)u_x.
\end{cases}
\end{equation}
Multiplying the first equation of \eqref{wpfc} by $u$ and the second equation of \eqref{wpfc} by $\theta_t$, and integrating by parts, we have 
\beq\label{en1}
\frac{\rho}{2}\frac{d}{dt}\int u^2\,dx=-\int g(\theta)u_x^2\,dx-\int h(\theta)\theta_tu_x \,dx,
\eeq
and 
 \begin{align}\label{en2}\begin{split}
\frac{1}{2}\frac{d}{dt}\int \left(\nu\theta_t^2+c^2(\theta)\theta_x^2\right)\,dx=&-\int \gamma_1\theta_t^2\,dx
-\int h(\theta)u_x\theta_t\,dx.
\end{split}
\end{align}
Sum up \eqref{en1} and \eqref{en2} to get
\begin{align*}
\frac{1}{2}\frac{d}{dt}\int \left(\nu\theta_t^2+ c^2(\theta)\theta_x^2+\rho u^2\right)\,dx=&-\int \left(\gamma_1\theta_t^2+2h(\theta)\theta_tu_x+g(\theta)u_x^2\right) \,dx\\
=& -\int_{\R}\Big(b(\theta)u_x^2+\gamma_1\big(\theta_t+\frac{1}{\gamma_1}h(\theta)u_x\big)^2\Big)\,dx,
\end{align*}
where $b(\theta)=g(\theta)-h(\theta)/{\gamma_1}>0$ is given in (\ref{btheta}).
 This completes the proof. 
\end{proof}

The term $-h(\theta)u_x$ in the second equation of system (\ref{time_poi}) could blow up (see Theorem \ref{sing}), and hence, is hard to control directly. In view of the special structures of the system, we will introduce new state and time variables.  

 We first make the following rescaling of time variable,   
 \begin{align*}
 \tilde{u}(x,t)=u(x,\sqrt{\nu}t)\;\mbox{ and }\;  \tilde{\theta}(x,t)=\theta(x,\sqrt{\nu}t).
 \end{align*}
 Then,   system (\ref{time_poi}) becomes 
 \begin{align}\label{time_poi_u1}\begin{split}
 \frac{\rho}{\sqrt{\nu}} \tilde{u}_t=&\Big( g(\tilde{\theta}) \tilde{u}_x+\frac{1}{\sqrt{\nu}} h(\tilde{\theta})\tilde{\theta}_t\Big)_x,\\
\tilde{\theta}_{tt}+\frac{\gamma_1}{\sqrt{\nu}} \tilde{\theta}_t=&c(\tilde{\theta})(c(\tilde{\theta})\tilde{\theta}_{x})_x
 - h(\tilde{\theta}) \tilde{u}_x.
 \end{split}
 \end{align}
 We then introduce a state variable   
  \begin{align*} 
 \tilde{v}(x,t)=\int_{-\infty}^x\rho\tilde{u}(z,t)dz.
 \end{align*}
Then, $\tilde{v}_t= \sqrt{\nu}g(\tilde\theta)\tilde{u}_x+h(\tilde{\theta})\tilde{\theta}_t$.
  One has  
 \begin{align} \label{time_poi_v1}\begin{split}
\frac{\rho}{\sqrt{\nu}} \tilde{v}_t=& g(\tilde{\theta}) \tilde{v}_{xx}+\frac{\rho} {\sqrt{\nu}}h(\tilde{\theta})\tilde{\theta}_t,\\
\tilde{\theta}_{tt}+\frac{1}{\sqrt{\nu}}\Big(\gamma_1-\frac{h^2(\tilde{\theta})}{g(\tilde{\theta})}\Big)\tilde{\theta}_t=&c(\tilde{\theta})(c(\tilde{\theta})\tilde{\theta}_{x})_x
 -\frac{h(\tilde{\theta})}{\sqrt{\nu}g(\tilde{\theta})}\tilde{v}_t.
 \end{split}
 \end{align}

 We emphasize that the relation 
 \begin{align}\label{damping}
 \gamma_1-\frac{h^2(\theta)}{g(\theta)}=\gamma_1\frac{b(\theta)}{g(\theta)}>0
 \end{align}
holds where $b(\theta)>0$ is defined in \eqref{btheta} and it gives a damping in $(\ref{time_poi_v1})_2$.
 
It turns out that the term 
\[\tilde{v}_t=\sqrt{\nu}g(\tilde\theta)\tilde{u}_x+h(\tilde{\theta})\tilde{\theta}_t=:\tilde J,\]
which is the flux density of the velocity from $\eqref{time_poi_u1}_1$,
captures the interaction between $u$ and $\theta$ well -- one has a good control on the $L^{\infty}$ norm of $\tilde J$. For smooth solutions, this is proved  in Lemma \ref{lemma3.1}.  It is much more involved to control the $L^{\infty}$ norm of $\tilde J$ for weak solutions. In replacing   $\tilde{u}_x$ in $(\ref{time_poi_u1})_2$ with   $\tilde{v}_t$ in $(\ref{time_poi_v1})_2$, we need some contribution from the damping term $\gamma_1\tilde\theta_t/{\sqrt{\nu}}$ in (\ref{time_poi_u1}). The feature that some damping is kept in $(\ref{time_poi_v1})_2$ due to (\ref{damping}) is crucial in the proof of global existence of weak solutions. See Lemma \ref{lemma3.3} in Section \ref{GER}.  

\medskip


%

\subsection{A special choice of parameters leading to system (\ref{simeqn0})} \label{Spara}
As mentioned in the introduction,  we consider a special case of Poiseuille flows \eqref{time_poi} in this paper; more precisely, we take   
\[
a=0,\quad \rho=\nu=1,\quad \alpha_1=\alpha_5=\alpha_6=0,\quad \alpha_2=-1,\quad \alpha_3=\alpha_4=1, \quad \gamma_1=2. 
\]
By the Onsager-Parodi relation (\ref{a2g}), one  has
\[
\gamma_2=\alpha_6-\alpha_5=\alpha_2+\alpha_3=0,
\]
and hence,  $g(\theta)=h(\theta)=1$.
System \eqref{time_poi} is then reduced  to \eqref{simeqn0}

This special case keeps the the main structure of \eqref{time_poi} while the heat equation $\eqref{time_poi}_1$
is simplified to one with constant coefficients, i.e. $g(\theta)=h(\theta)=1$.   We do believe  that similar singularity formation and global existence results in this paper hold true for \eqref{time_poi}.

In terms of the variable $(v,\theta)$,  where 
 \begin{align}\label{u2v}
  v(x,t)=\int_{-\infty}^x u(z,t)dz, \;\mbox{ and hence, }\; v_t=u_x+\theta_t,
 \end{align}    
 system \eqref{simeqn0}  reads
\begin{align}\label{simeqn} \begin{split}
 v_t=&v_{xx}+\theta_t,\\ 
\theta_{tt}+\theta_t=&c(\theta)(c(\theta)\theta_{x})_x
-v_t.
\end{split}
\end{align}
 
   We remind the readers that the damping term $\theta_t$ in the second equation is not due to the special choice of the parameters but the intrinsic property of the problem discussed immediately after display (\ref{time_poi_v1}).

  For system \eqref{simeqn0}, the energy introduced in \eqref{E.0} simplifies to \eqref{E.00},
which,   for smooth solutions,  satisfies
\[\frac{d}{dt}\mathcal{E}(t)= -\int_{\R}\Big(\frac{1}{2} v_{xx}^2+2\Big( \theta_t+\frac{1}{2}v_{xx}\Big)^2\Big)\;dx=-\int_{\R}(v_t^2+\theta^2_t)\,dx.\]
  The latter shows our result in statement (i) of Theorem \ref{thmplc1} on energy estimates for weak solutions is sharp.




\subsection{Main ideas of proofs}\label{MI}
One of key contributions of this paper is  the identification of the crucial quantity
\[J:=v_t=u_x+\theta_t,\]
where  $v(x,t)=\int_{-\infty}^{x}u(z,t)\,dz$ introduced  in (\ref{u2v}).

For any time $T$, it will be shown  that  $J(x,t)$ is defined for $(x,t)\in\mathbb R\times [0,T]$ and has finite $L^2$, $L^\infty$ and $C^\alpha$ norms, with $\alpha\in(0, 1/4)$.  Given the fact that both $u_x$ and $\theta_t$ may blow up (Theorem \ref{sing}), the bound and regularity of $J$ are fundamentally important. It will be seen that the result is indeed critical for both singularity formation and global existence. Roughly speaking, this result holds because of the different ``scales'' of time variable $t$ in heat equation and in wave equation. To understand it, one can first look at Lemma \ref{lemma3.1} that gives a bound on $J$ associated with smooth solutions.


In our construction of the example with cusp formation, we adapt the framework in \cite{CHL, CZ12} to our coupled system.
For smooth solutions from our construction, the bound of $J$ can be carefully estimated using the initial energy using Lemma \ref{lemma3.1}. Especially, such a bound is small when $u_0$ and $\mathcal E(0)$ are both small, although $\theta_t(x,0)$ and $\theta_x(x,0)$ might be large near the point of singularity formation. In this case, the compressive effect from the quasilinear wave equation dominates the dissipative effect from the heat equation and leads to a cusp singularity in finite time.

The global existence part is much more complicated.  In the first step, for a given  bounded, square integrable and H\"older continuous function $J$, we replace $v_t$ with $J$ in equation $\eqref{simeqn}_2$ and solve for $\theta=\theta^J(x,t)$. Instead of considering the problem directly in the $(x,t)$ coordinates, we start the analysis from an equivalent semilinear system in the characteristic coordinates and, afterward, we transform back to the $(x,t)$ coordinates. This framework was used in \cite{BZ} for variational wave equation. 

Here we mention two differences from this paper to \cite{BZ}. First, due to the low regularity of $J(x,t)$, we use Schauder fixed point theorem to prove the existence of solution on characteristic coordinates.
Secondly, since there is no direct way to control two key dependent variables $p$ and $q$, which measure the dilation of the transformation,  by the semilinear system, a great deal of extra efforts are made in finding the a priori bounds on $p$ and $q$ using the relation \eqref{damping} which works even for the general case. In fact, by \eqref{damping}, we know the nematic liquid crystal model naturally gives us some ``leftover'' damping after we change from $\eqref{simeqn0}_2$ to $\eqref{simeqn}_2$, and such a ``leftover'' damping term plays a crucial role in bounding $p$ and $q$. See Lemma \ref{lemma3.3} for this key estimate.


In the second step, using the heat equation $\eqref{simeqn}_1$, with  $\theta(x,t)=\theta^J(x,t)$, we solve for $v=v^J(x,t)$. 
This allows us to define a map $J\to v^J_t$. A fixed point of this map using the Schauder Fixed-Point Theorem gives the existence. 

To show the energy decay, we need to conquer the ``mismatch'' between the semilinear system in characteristic coordinates and the heat equation.

%
 
\section{Estimates on $J=v_t$ for smooth solutions}\label{Lfun}
In this section,  we derive some estimates on $J=v_t$ for any smooth solution of \eqref{simeqn}. 
Recall from \eqref{u2v}-\eqref{simeqn} that 
\beq\label{sec3Lv}
v(x,t)=\int_{-\infty}^x
u(z,t)\,dz\;\mbox{ and }\; v_t=v_{xx}+\theta_t=u_x+\theta_t.
\eeq

\begin{lemma}\label{lemma3.1}
If $(v,\theta)$ is a smooth solution of system \eqref{simeqn} for $t\in [0,T)$, one has, for any $t\in [0,T)$,
\beq\label{vtbds}
\begin{split}
\|v_t\|_{L^\infty(\R\times(0,t))} \leq& \|u'_0(x)+\theta_1(x)\|_{L^{\infty}(\R)}\\
&+Ct^{\frac14}\left(\|\theta_t\|_{L^\infty((0,t), L^2(\R))}+\|\theta_x\|_{L^\infty((0,t), L^2(\R))}\right)\\
&+Ct^{\frac14}\left(\|\theta_x\|^2_{L^\infty((0,t), L^2(\R))}+t^{\frac14}\|u\|_{L^\infty(\R\times(0,t))}\right),
\end{split}
\eeq
and 
\beq\label{ubds}
\|u\|_{L^\infty(\R\times(0,t))}=\|v_x\|_{L^\infty(\R\times(0,t))}\leq \|u_0\|_{L^{\infty}(\R)}+Ct^{\frac14}\|\theta_t\|_{L^\infty((0,t), L^2(\R))}.
\eeq

\end{lemma}

\begin{proof}
Recall that
\beq\label{Hdef}
H(x,t)=\frac{1}{\sqrt{4\pi t}}\, \exp\left\{-\frac{x^2}{4t}\right\}
\eeq
is the fundamental solution of 1D heat equation, that is, 
\[H_t-H_{xx}=0,\quad \mbox{for } (x,t)\in\R\times(0,\infty)\]
with  $H(x,0)=\delta_0(x)$ being the Dirac function at $x=0$.

We decompose 
\beq\label{sec5lk}
v(x,t)=l(x,t)+k(x,t)
\eeq
where $l(x,t)$ and $k(x,t)$ satisfy, respectively, 
\beq\label{inf1}
\left\{
\begin{array}{ll}
\ds l_t=l_{xx}+\theta_t,\\
\ds l(x,0)=0,
\end{array}
\right.
\eeq 
and  
\beq\label{inf2}
\left\{
\begin{array}{ll}
\ds k_t=k_{xx},\\
\ds k(x,0)=\int_{-\infty}^xu_0(y)\,dy:=k_0(x).
\end{array}
\right.
\eeq
By the Duhamel formula, one has
\[l(x,t)=\int_0^t\int_{\mathbb R}H(x-y,t-s)\theta_s(y,s)\,dyds=\int_0^t\int_{\mathbb R}H(x-y,s)\theta_s(y,t-s)\,dyds,\]
and
 \beq\label{kexp}
\begin{split}
k(x,t)&=\int_{\mathbb R}H(x-y,t)k_0(y)\,dy\\
&=k_0(x)+\int_0^t\int_{\mathbb R}H(x-y,t-s)\big(k_0\big)_{yy}(y)\,dyds\\
&=k_0(x)+\int_0^t\int_{\mathbb R}H(x-y,s)\big(k_0\big)_{yy}(y)\,dyds.
\end{split}
\eeq
Direct calculation implies
\beq\label{key}
\begin{split}
&|l_x(x,t)|\\
\leq& C\int_0^t\int_{\R}\frac{|x-y|}{(t-s)^{3/2}}\exp\left(-\frac{(x-y)^2}{4(t-s)}\right)|\theta_s(y,s)|\,dyds\\
\leq& C\left(\int_0^t\int_{\R}\frac{|x-y|^2}{(t-s)^{3/2+3/4}}\exp\left(-\frac{(x-y)^2}{2(t-s)}\right)\,dyds\right)^{\frac12}\cdot\left(\int_0^t\int_{\R}\frac{|\theta_s(y,s)|^2}{(t-s)^{3/4}}\,dyds\right)^{\frac12}\\
\leq& C\|\theta_t\|_{L^\infty((0,T), L^2(\R))}\int_0^t\frac{1}{(t-s)^{3/4}}\,ds
\leq Ct^\frac{1}{4}\|\theta_t\|_{L^\infty((0,T), L^2(\R))}.
\end{split}
\eeq
On the other hand, one can show
\beq
\begin{split}
|k_x(x,t)|
=\left|\int_{\R}H(y,t)u_0(x-y)\,dy\right|
\leq \|u_0\|_{L^{\infty}(\R)},
\end{split}
\eeq
which combining with \eqref{key} implies \eqref{ubds}.

\medskip
To obtain the estimate of $v_t$, by the definition of $k$ and $l$, we have 
\beq\notag
\begin{split}
v_t=&l_t+k_t\\
=&\frac{d}{dt}\int_0^t\int_{\mathbb R}H(x-y,s)\theta_s(y,t-s)\,dyds+\frac{d}{dt}\int_0^t\int_{\R}H(x-y,s)k_0''(y)\,dyds\\
=&\int_{\mathbb R}H(x-y,t)\left(\theta_1(y)+u'_0(y)\right)\,dy+\int_0^t\int_{\mathbb R}H(x-y,s)\theta_{st}(y,t-s)\,dyds\\
=&\int_{\mathbb R}H(x-y,t)\left(\theta_1(y)+u'_0(y)\right)\,dy-\int_0^t\int_{\mathbb R}H(x-y,s)\theta_{ss}(y,t-s)\,dyds.
\end{split}
\eeq
The first term can be estimated as follows
\beq\label{estkt}
\begin{split}
\left|\int_{\mathbb R}H(x-y,t)\left(\theta_1(y)+u'_0(y)\right)\,dy\right|
\leq \|u'_0(x)+\theta_1(x)\|_{L^{\infty}(\R))}.
\end{split}
\eeq
By $\eqref{simeqn}_2$, we can rewrite the second term as
\beq\label{estlt_last}
\begin{split}
\int_0^t\int_{\mathbb R}&H(x-y,s)\theta_{ss}(y,t-s)\,dyds\\
=&\int_0^t\int_{\mathbb R}H(x-y,s)\big[-2\theta_s+c(\theta)(c(\theta)\theta_y)_y-u_y\big](y,t-s)\,dyds\\
=&\int_0^t\int_{\mathbb R}H(x-y,s)\big[-2\theta_s-c'(\theta)c(\theta)\theta_y^2\big](y,t-s)\,dyds\\
&+\int_0^t\int_{\mathbb R}H_y(x-y,s)\big[c^2(\theta)\theta_y-u\big](y,t-s)\,dyds.
\end{split}
\eeq
Similarly to \eqref{key}, one can show that
\beq\label{estlt2}
\begin{split}
&\left|\int_0^t\int_{\mathbb R}H(x-y,s)\theta_{ss}(y,t-s)\,dyds\right|\\
&\quad \leq Ct^{\frac14}\left(\|\theta_t\|_{L^\infty((0,t), L^2(\R))}+\|\theta_x\|_{L^\infty((0,t), L^2(\R))}\right)\\
 &\qquad+Ct^{\frac14}\left(\|\theta_x\|^2_{L^\infty((0,t), L^2(\R))}+t^{\frac14}\|u\|_{L^\infty(\R\times(0,t))}\right).
\end{split}
\eeq

Note in \eqref{estlt_last}, we use the different ``scales'' of time variable $t$ in heat equation and in wave equation. Heuristically, we use the wave equation to change $H*\theta_{tt}$ to $H*\theta_{xx}$ and other lower order terms. This helps us bound the term in \eqref{estlt_last} that leads to the bound of $J$.

Combining \eqref{estkt} and  \eqref{estlt2}, one has  \eqref{vtbds}.
\end{proof}

By the energy decay proved  in  Proposition \ref{E4smooth}
\[
\mathcal{E}(t)\leq \mathcal{E}_0\quad \hbox{for any}\quad 0\leq t\leq T,
\]
and also using Lemma \ref{lemma3.1},
we know that there exists a constant $\hat{J}(T)$  depending only on $\|u_0'\|_{L^2(\R)}$, $\mathcal{E}_0$ and $T$, such that,
\beq\label{L00}
\max_{(x,t)\in \mathbb{R}\times [0,T]}{|J(x,t)|}<\hat{J}(T).
\eeq

From the proof of Lemma \ref{lemma3.1},  
\beq\label{hatu}
\begin{split}
 u(x,t) =&\int_{\R}H(x-y,t)u_0(y)\,dy+\int_0^t\int_{\R}H_x(x-y,t-s)\theta_{s}(y,s)\,dyds.
\end{split}
\eeq
One then has
\beq\label{hatvt}
\begin{split}
J(x,t)
=&\int_{\R}H(x-y,t)\left(u_0'(y)+\theta_1(y)\right)\,dy\\
&+\int_0^t\int_{\mathbb R}H(x-y,t-s)\big[2\theta_s+c'(\theta)c(\theta)\theta_y^2\big](y,s)\,dyds\\
&-\int_0^t\int_{\mathbb R}H_y(x-y,t-s)\big[c^2(\theta)\theta_y-u\big](y,s)\,dyds.
\end{split}
\eeq
These two relations will be used in Section \ref{sec_6.2} where we prove the global existence of weak solutions. In fact, the main step in the proof of existence is to find a fixed point of a map $\mathcal M (J)$, constructed by using \eqref{hatu} and \eqref{hatvt}, for $J$ in $L^\infty\cap L^2$. By \eqref{L00} and \eqref{sec3Lv}, one can easily see why we use the sup-norm space and square integrable function space, respectively.
Secondly, the estimate on $v_t=J$ in Lemma \ref{lemma3.1} and the energy decay will give us the key estimate on $J$ for the singularity formation  in the next section.




\section{Singularity formation}\label{SF}
\setcounter{equation}{0}

This section is devoted to a proof of  Theorem \ref{sing} for the singularity formation.
We extend the   framework for the variational wave equation in  \cite{GHZ,CZ12} to the coupled system \eqref{simeqn0}. The main difference is that there is a nonlocal source term $J$ in the second equation of \eqref{simeqn0}. The estimate on $J$ in Lemma \ref{lemma3.1} is thus the major new ingredient for our construction. 

 The proof of Theorem \ref{sing} is split into several steps.     We will show that, if $\ve$ is small enough, then the singularity will appear before $t=1$. 
Thus all estimates below are for solutions over $0\leq t<1$.

\medskip
\noindent{\em Step 1.} For any smooth solution $(u(x,t),\theta(x,t))$ of \eqref{simeqn0},   set  
\begin{align}\label{S_R_def}
 S := \theta_t - c(\theta) \theta_x,\quad
 R := \theta_t+c(\theta) \theta_x.
 \end{align}
It follows from $\eqref{simeqn0}_2$ that
 \begin{align}\label{RxSx}\begin{split}
 S_t+c(\theta)S_x =& \frac{c'(\theta)}{4c(\theta)} (S^2-R^2)-(R+S)-u_x,\\
 R_t -c(\theta) R_x =& \frac{c'(\theta)}{4c(\theta)} (R^2-S^2)-(R+S)-u_x,
\end{split}
\end{align}
or, with $J=v_t=u_x+\theta_t$ as in the previous section,
\begin{align}
S_t+c(\theta)S_x =& \frac{c'(\theta)}{4c(\theta)} (S^2-R^2)-\frac{1}{2}(R+S)-J,\label{S_sing}\\
 R_t -c(\theta) R_x =& \frac{c'(\theta)}{4c(\theta)} (R^2-S^2)-\frac{1}{2}(R+S)-J.
 \label{R_sing}
\end{align}
It is then easy to have
\beq\label{RS_sing}
\bigl(R^2+S^2\bigr)_t+\bigl(c(\theta)(S^2-R^2)\bigr)_x=-(S+R)^2-2(S+R)J.
\eeq

\medskip
\noindent{\em Step 2.}
From the initial condition \eqref{id_sing0} set for Theorem \ref{sing}, one has 
\begin{align}\label{R_0_sing}
\theta_x(x,0)=\phi'(\frac{x}{\e}),\;
R(x,0)=\e\, \phi'(\frac{x}{\e}),\;
S(x,0)=\big(-2c(\theta(x,0))+\e \big)\, \phi'(\frac{x}{\e}). 
\end{align}
We always choose $0<\e<C_L$ and $\e\ll 1$. Here recall from \eqref{condc}  that the uniform lower and upper bounds of the function $c$ are $C_L$ and $C_U$, respectively.
 So by \eqref{id_sing2},
\beq\label{S00}
S(0,0)=(-2c(\theta^*)+\ve)\phi'(0)>\max\left\{\frac{16C_U}{c'(\theta^*)},2\right\}.
\eeq

We now define a function $E$ as
\beq\label{E_def}
E(t)\equiv E(\theta(\cdot, t))=\int_{-\infty}^{\infty}(\theta^2_t+c^2\theta^2_x)(x, t)\,dx
=\frac12\int_{-\infty}^{\infty} (R^2+S^2)(x, t)\, dx.
\eeq
For smooth solutions, by the energy decay proved in Proposition \ref{E4smooth}, and properties in \eqref{id_sing3}-\eqref{asmPhi2}, we have 
\begin{align}\label{E_est}\begin{split}
E(t) \leq& 2\mathcal{E}(t)\leq2\mathcal{E}(0)=   \frac12\int_{-\infty}^{\infty} (R^2+S^2+2u^2)(x, 0)\, dx\\
	 =&   \frac12\int_{-\infty}^{\infty} \big[\,\left( - 2\,c(\theta(x,0))+\e \right)^2+\e^2\,\big]\, 
	\big(\phi'(\frac{x}{\e}) \big)^2 \, dx +KC_U^2C_2^2\ve^2 \leq  O(\e),
\end{split}
\end{align}
where $K$ is some constant,
and $\ve$ is any sufficiently small constant.
It then follows from Lemma \ref{lemma3.1}, \eqref{id_sing0}-\eqref{asmPhi2} and \eqref{E_est} that there exists a constant $k_1$ independent of $\ve$ such that, 
\beq\label{J_est} 
\| J\|_{L^\infty(\mathbb R\times(0,t))}\leq k_1\sqrt\e\quad \hbox{for}\quad 0\leq t\leq 1.
\eeq



\begin{figure}[htb]
\centering
\includegraphics[scale=.4]{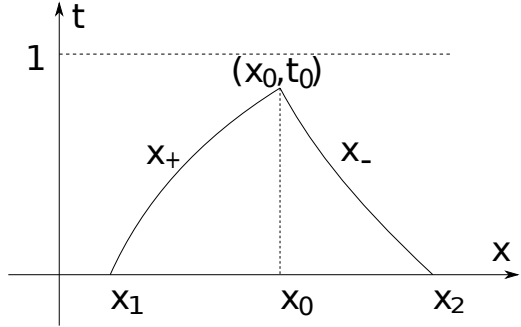}
\caption{Characteristic triangle $\Omega$}
\label{f0}
\end{figure}

\medskip
\noindent{\em Step 3.}
Next we consider any characteristic triangle  $\Omega$  in Figure \ref{f0} 
bounded by the $x$-axis together with the characteristic curves $x_{\pm}(t)$ (or $t_{\pm}(x)$) given by 
\[\frac{dx_{\pm}(t)}{dt}=\pm c(\theta)\leq C_U,\;\mbox{ or equivalently, }\; 
\frac{dt_{\pm}(x)}{dx}=\pm \frac{1}{c(\theta)}.
\]
It is easy to see that, for any $(x_0,t_0)$ with $t_0<1$, 
\beq\label{xpm}
|x_2-x_1|<2C_U.
\eeq
Integrating \eqref{RS_sing} over  $\Omega$ and applying the divergence theorem, we have,
\begin{align*}
\int_{x_1}^{x_0}R^2(x, t_{+}(x))\,dx+&\int_{x_0}^{x_2}S^2(x,t_{-}(x))\,dx = 
\frac12 \int_{x_1}^{x_2} \big(R^2(x,0)+S^2(x,0)\big)\,dx\\
  +&\frac{1}{2}\iint_{\Omega}(S+R)^2\,dx\,dt+\iint_{\Omega}(S+R)J\,dx\,dt.
 \end{align*}
 Using \eqref{E_est}, \eqref{J_est} and \eqref{xpm}, one has, for $\ve$ small,
\begin{align}\label{dependence}\begin{split}
&\int_{x_1}^{x_0}R^2(x, t_{+}(x))\,dx+\int_{x_0}^{x_2}S^2(x,t_{-}(x))\,dx\\
 &\qquad\leq 
\frac12 \int_{x_1}^{x_2} \left[R^2(x,0)+S^2(x,0)\right]\,dx+\int_0^{t_0}\int_{x_-(t)}^{x_+(t)}(S^2+R^2)\,dx\,dt\\
&\qquad\qquad +\iint_{\Omega}(|S|+|R|)|J|\,dx\,dt\\
&\qquad <12C_U^2k_0\ve+k_1\sqrt\ve \int_{0}^t\int_{x_-(t)}^{x_+(t)}(|S|+|R|)\,dx\,dt\\
&\qquad <12C_U^2k_0\ve+k_1\sqrt{\ve}(2C_U)^\frac{1}{2}(4C^2_U k_0\ve)^\frac{1}{2} =k_2\ve
\end{split}
\end{align}
where $k_2= 12 k_0C_U^2+2\sqrt{2 k_0} k_1C_U^{\frac{3}{2}}.$

\medskip
\noindent{\em Step 4.}
Consider now the forward characteristic piece $x=\Gamma(t)$ for $t\in [0,1]$ starting from the origin,   that is, 
\[
\frac{d\Gamma(t)}{dt}=c\Bigl(\theta\bigl(\Gamma(t),t\bigr)\Bigr),\quad \Gamma(0)=0.
\]
We will show the singularity formation by tracking   $S(t)\equiv S(\Gamma(t),t)$ along $\Gamma$. 

We know that
\[
\frac{d \theta(\Gamma(t),t)}{dt}=R(\Gamma(t),t).
\]
Integrate this equation and use \eqref{dependence} to have
\begin{align}\begin{split}
&|\theta(\Gamma(t),t)-\theta(\Gamma(0),0)|=|\int_{0}^t R(\Gamma(t),t)\, dt|\nn\leq  \sqrt{t}\sqrt{\int_{0}^t R^2(\Gamma(t),t)\, dt}\nn\\
&\qquad \leq  \sqrt{\int_{x_0}^{x} R^2(x,t(x))\frac{1}{c}\, dx}\nn\leq  \sqrt{\frac{1}{C_L}\int_{x_0}^{x} R^2(x,t(x))\, dx}\leq  \sqrt{\frac{k_2}{C_L}\ve},
\end{split}
\end{align}
 where without confusion we use $t(x)$ to denote the characteristic $\Gamma(t)$. Recall that $c$ is $C^2$. Thus, if $\e$ is small enough, one has 
\beq\label{cd_sign}
c'(\theta(\Gamma(t),t))>\frac{c'(\theta(\Gamma(0),0))}2=\frac{c'(\theta^*)}2>0.
\eeq

We claim that before time $t=1$, if there is no break down of classical solution, then 
$S(\Gamma(t),t)>1$. This will be proved   by contradiction. Suppose that $t_*$ is the first time such that 
\beq\label{key_4}S(\Gamma(t_*),t_*)=1,\eeq 
while
\beq\label{key_5}
S(\Gamma(t),t))>1 \quad\hbox{for}\quad 0<t<t_*\leq 1. 
\eeq
Now we only consider $\Gamma(t)$ with $ 0<t<t_*\leq 1$.
Set $\tilde S=e^{\frac{1}{2}t} S$.
By \eqref{S_sing}, we have
\[
{\tilde S}_t+c {\tilde S}_x
= \frac{c'}{4c}e^{-\frac{1}{2}t}  {\tilde S}^2 -\frac{c'}{4c}e^{\frac{1}{2}t}R^2-\frac{1}{2}e^{\frac{1}{2}t} R - e^{\frac{1}{2}t} J.
\]
Along the characteristic $\Gamma(t)$, we have
\[
\frac{d}{dt}\tilde S
\geq \frac{c'}{4c}e^{-\frac{1}{2}t}  {\tilde S}^2 -\frac{c'}{4c}e^{\frac{1}{2}t}R^2-\frac{1}{2}e^{\frac{1}{2}t}|R|-e^{\frac{1}{2}t}|J|.
\]
Divide the above by $\tilde S^2$ and  integrate to get 
\begin{align}\label{7.18}\begin{split}
\frac{1}{\tilde S(t_*)}&\leq \frac{1}{\tilde S(0)}-\frac{c'(\theta^*)}{8C_U}t_*+\int_0^{t_*}\frac{1}{\tilde S^2}\left\{\frac{c'}{4c}e^{\frac{1}{2}t}R^2+\frac{1}{2}e^{\frac{1}{2}t}|R|+e^{\frac{1}{2}t}|J|\right\}dt\\
&\leq \min\{ \frac{c'(\theta^*)}{16C_U},\frac{1}{2}\} -\frac{c'(\theta^*)}{8C_U} t_*+k_3\sqrt{\ve},
\end{split}
\end{align}
for some positive constant $k_3$ independent of $\ve$ because of \eqref{dependence},  \eqref{key_5} and \eqref{J_est}, where we also use \eqref{cd_sign} and \eqref{S00}. If $\ve$ is small enough, one has  
$\tilde{S}(t_*)>e^{\frac{1}{2}}$,   and hence,    $S(t_*)>1$, which contradicts to \eqref{key_4}.

Hence, if there is no blowup before $t=1$, then  $S(\Gamma(t),t))>1$ for $0<t<1$.

\medskip
\noindent{\em Step 5.}
Finally, we prove the breakdown of the solution. By the same calculation as in \eqref{7.18},
\[
\frac{1}{\tilde S(t)}\leq  \frac{c'(\theta^*)}{16C_U} -\frac{c'(\theta^*)}{8C_U} t+k_3\sqrt{\ve}.
\]
Therefore  $\tilde S(t)\rightarrow \infty$ will occur no later than the time when the right-hand side is zero or
\[
t=\frac{1}{2}+\frac{8C_Uk_3} {c'(\theta^*)}\sqrt{\ve}<1,
\]
where the inequality holds when $\ve$ is   small enough. This completes the proof of Theorem \ref{sing}. 

\begin{figure}[htb]
\centering
\includegraphics[scale=.35]{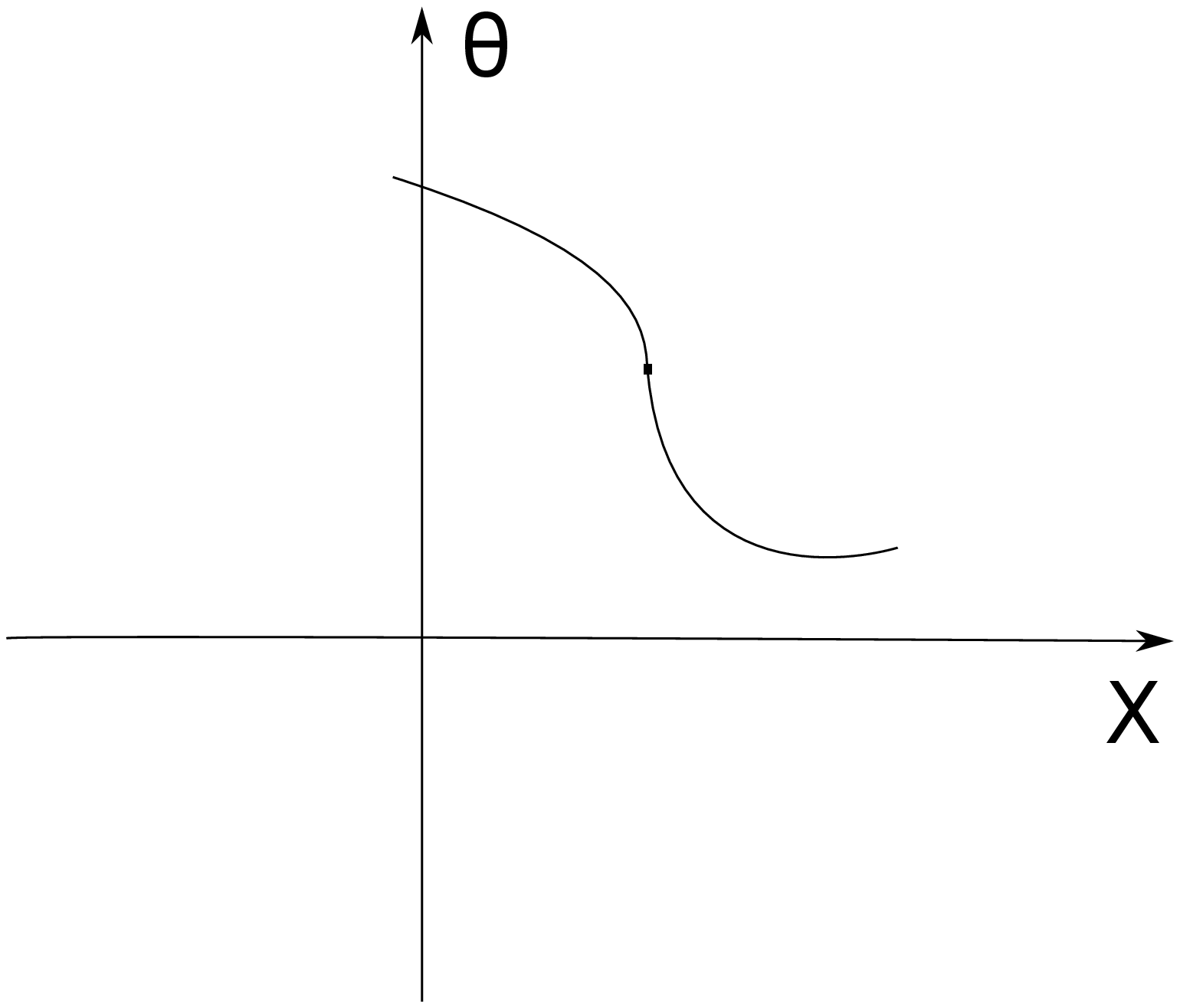}
\includegraphics[scale=.35]{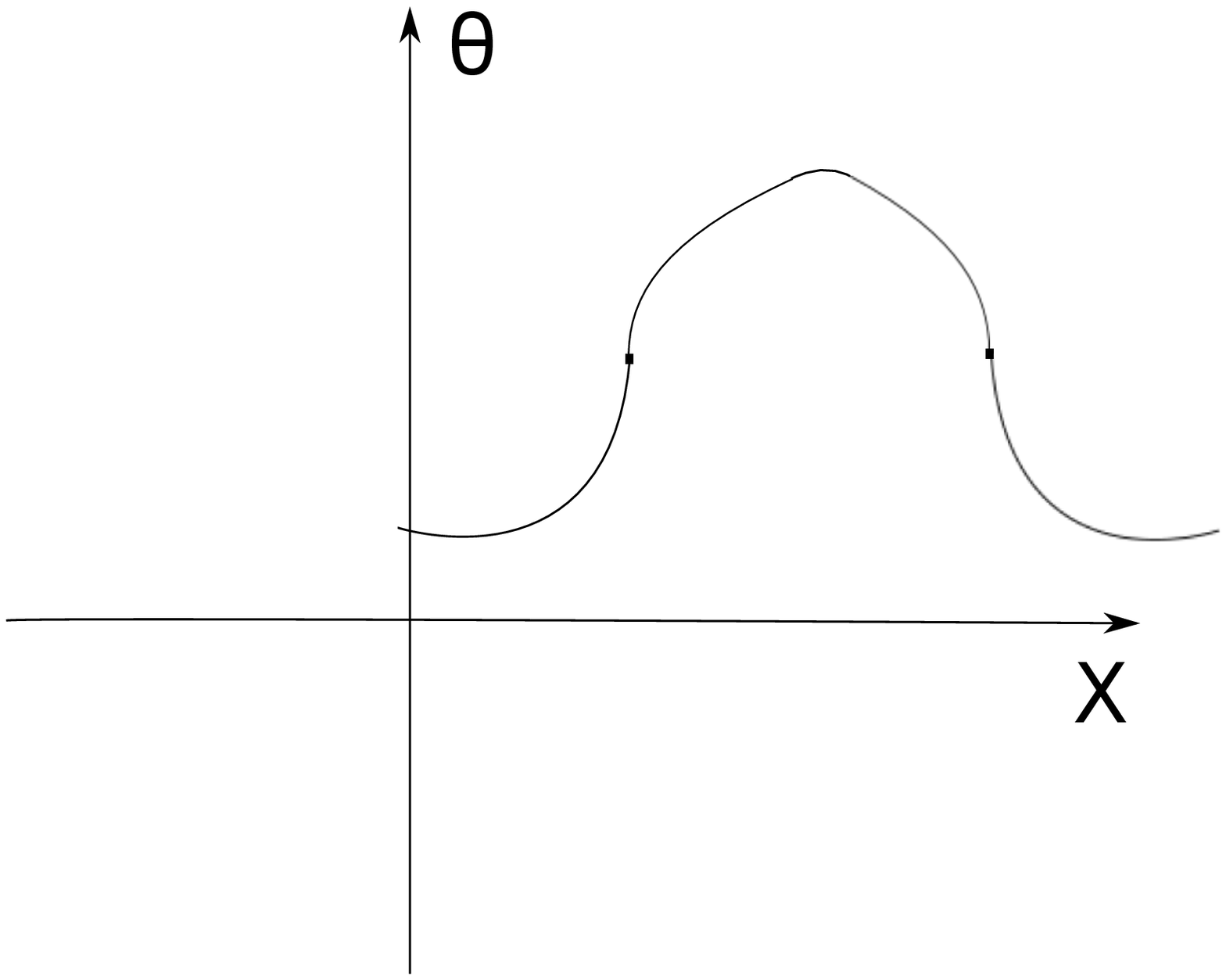}
\includegraphics[scale=.35]{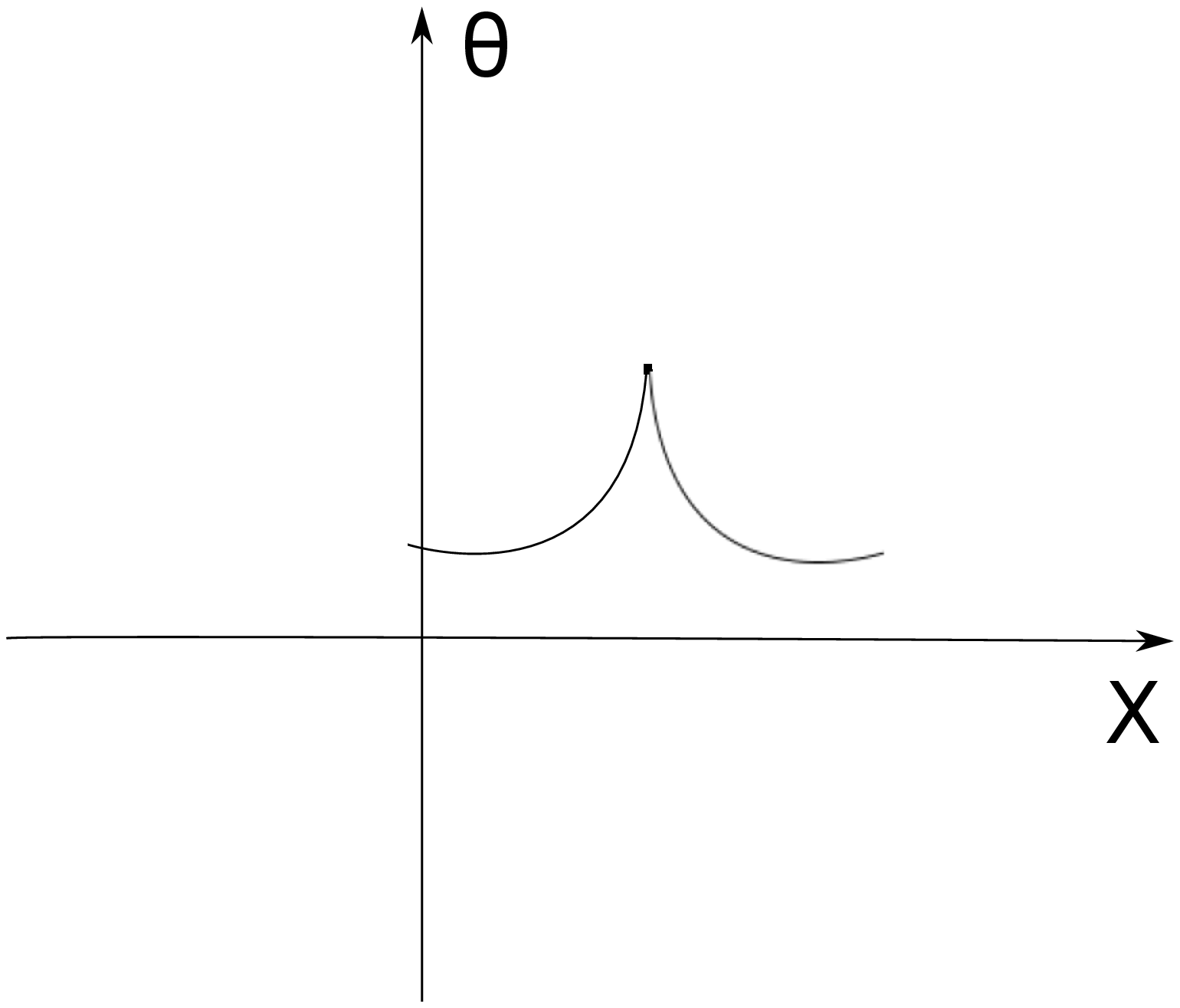}
\caption{{\em The first picture illustrates a one-sided cusp proved in Theorem \ref{sing}. The third picture illustrates a full cusp singularity.}}
\label{fig567}
\end{figure}

We close this section with  a remark on why the singularity is a cusp singularity.
\begin{remark}\label{rem}
By \eqref{R_0_sing}, we know that the maximum initial value of $|R(x,0)|$ is of $O(\ve)$ order. It is easy to get that $R$ will be bounded above in any $O(1)$ time by studying the Riccati equation $\eqref{RxSx}_2$ and using $c'>0$ which can be proved similarly as \eqref{cd_sign} when $\ve$ is small enough. Thus, at the point of blowup, one has 
\[
S=\theta_t-c\theta_x=\infty,\quad |R|=|\theta_t+c\theta_x|<\hbox{constant},
\]
which imply that $\theta_t=\infty$ and $\theta_x=-\infty$. The singularity is typically one-sided cusp. For carefully designed initial data, two opposite one-sided cusps might occur at the same time and the same location to form a full cusp.
See Figure \ref{fig567}.

 Together with the energy decay for smooth solutions, we know that the singularity formed in finite time is a cusp singularity with  derivatives $|\theta_x|$ and $|\theta_t|$ being infinity, but the $L^2$ norms of which are finite. Hence $\theta$ is H\"older continuous with exponent $1/2$. By the estimate on $v_t$ in Lemma \ref{lemma3.1} and   $u_x=v_t-\theta_x$, we also know that $u(\cdot,t)$ is H\"older continuous with exponent $1/2$ before and at the blowup.
\end{remark}

\section{A semilinear system in characteristic coordinates}\label{SLS}

As commented in Section \ref{sec-results} on the approach for global existence result, we will rewrite system \eqref{simeqn}  into a semilinear system in characteristic coordinates.   

For any smooth solution   $(u(x,t),\theta(x,t))$ of \eqref{simeqn0}, the equations of the characteristics are
\beq\label{char}
\frac{d x^\pm(s)}{d s}=\pm c\bigl(\theta(x^\pm(s),s)\bigr),
\eeq
where, at time $s$,
\[x^\pm(s)\equiv x^\pm(s;\, x,\, t)\]
denote the forward and backward characteristics passing through the point $(x,t)$, respectively.
Using the variables
 \[
 S := \theta_t - c(\theta) \theta_x,\quad
 R := \theta_t+c(\theta) \theta_x,
\]
 defined  in (\ref{S_R_def}),
we introduce new coordinates $(X, Y)$ by
\begin{align}\label{charaXY}\begin{split}
X\equiv X(x,t):=&	\int_1^{x^-(0;\, x,\, t)} (1+ R^2(x',0))\,d x',\\
 Y\equiv Y(x,t):=&	\int_{x^+(0;\, x,\, t)}^1 (1+ S^2(x',0))\,d x'.
 \end{split}
\end{align}
It is easy to check that $X$ and $Y$ are constants along backward and forward characteristic, respectively; that is, 
\beq\label{X_Y_eqn}
 X_t -c(\theta) X_x=0 \com{and}  Y_t +c(\theta) Y_x=0.
\eeq
Here we use $1+ R^2$ and $1+ S^2$ as the integrands in \eqref{charaXY} just for later convenience in assigning the boundary data. One could choose other nonzero integrable functions as the integrands.
For any smooth function $f$, it follows from (\ref{charaXY}) that  
\beq\left\{\begin{array}{l}\label{trans}
f_t +c(\theta) f_x=f_X\,(X_t +c(\theta) X_x) +f_Y\, (Y_t +c(\theta) Y_x)=2c X_x f_X,\\
f_t -c(\theta) f_x=f_X\,(X_t -c(\theta) X_x) +f_Y\, (Y_t -c(\theta) Y_x)=-2c Y_x f_Y.
\end{array}\right.
\eeq

In order to complete the system, we introduce several variables: 
\beq\label{defwz}
w=2\arctan R,\qquad z=2\arctan S,
\eeq
and 
\beq\label{defpq}
p=\frac{1+R^2}{X_x},\qquad q=\frac{1+S^2}{-Y_x}.
\eeq
Furthermore,
set $f=x$ and $f=t$ in two equations in \eqref{trans} to get
\beq\label{cxtXY}
\left\{
\begin{array}{rcl}
\ds c&=&2cX_x x_X,\\
\\
\ds -c&=&-2cY_x x_Y,
\end{array}
\right.
\quad 
\left\{
\begin{array}{rcl}
\ds 1&=&2cX_x t_X,\\
\\
\ds 1&=&-2cY_x t_Y.
\end{array}
\right.
\eeq
Using \eqref{defwz} and \eqref{defpq}, it holds
\beq\label{eqnxt}
\left\{
\begin{array}{ll}
\ds x_X=\frac{1}{2X_x}=\frac{1+\cos w}{4}p,\\
\\
\ds x_Y=\frac{1}{2Y_x}=-\frac{1+\cos z}{4} q,
\end{array}
\right.
\qquad 
\left\{
\begin{array}{ll}
\ds t_X=\frac{1}{2cX_x}=\frac{1+\cos w}{4c}p,\\
\\
\ds t_Y=-\frac{1}{2cY_x}=\frac{1+\cos z}{4c} q.
\end{array}
\right.
\eeq

Then  system \eqref{simeqn} can be written as follows.
 \beq\label{semilinear_new}
\begin{split}
\theta_X&=\frac{\sin w}{4c}p, \quad \theta_Y=\frac{\sin z}{4c}q,\\
z_X&=\frac{p}{4c}
\Big\{\frac{c'}{c}(\cos^2\frac{w}{2}-\cos^2\frac{z}{2})-\sin w\cos^2\frac{z}{2}-\sin z\cos^2\frac{w}{2}-4J\cos^2\frac{z}{2}\cos^2\frac{w}{2} \Big\}, \\
w_Y&=\frac{q}{4c}
\Big\{\frac{c'}{c}(\cos^2\frac{z}{2}-\cos^2\frac{w}{2})-\sin w\cos^2\frac{z}{2}-\sin z\cos^2\frac{w}{2}-4J\cos^2\frac{z}{2}\cos^2\frac{w}{2}\Big\}, \\
p_Y&=\frac{pq}{2c}\Big\{\frac{c'}{4c}(\sin z-\sin w)
-\frac{1}{4}\sin w\sin z-\sin^2 \frac{w}{2}\cos^2\frac{z}{2}
-J\sin w\cos^2\frac{z}{2}\Big\},\\
q_X&=\frac{pq}{2c}\Big\{\frac{c'}{4c}(\sin w-\sin z)
-\frac{1}{4}\sin w\sin z-\sin^2 \frac{z}{2}\cos^2\frac{w}{2}
-J\sin z\cos^2\frac{w}{2}\Big\}
\end{split}
\eeq
where, recalling that,  
\beq\label{J_def}
J=u_x+\theta_t=u_x+\frac{S+R}{2}.
\eeq
A derivation of the semilinear system \eqref{semilinear_new}-\eqref{J_def}  is given in  Appendix \ref{Sec_A3}. It is a special case of  system  \eqref{semilinearA} that is derived from \eqref{time_poi}.
 
\begin{remark}We observe that the new system is invariant under translation by $2\pi$ in $w$ and
$z$. Actually, it would be more precise to work with the variables $\hat w=
e^{iw}$ and $\hat z=
e^{iz}$.
However, for simplicity we shall use the variables $w$, $z$, keeping in mind that they range
on the unit circle $[-\pi,\pi]$ with endpoints identified.
\end{remark}

\section{Global existence and energy decay}\label{GER}
In this section, we prove the global existence and energy decay in Theorem \ref{thmplc1}. Here by global existence, we mean that, for any $T>0$, there exists a solution for  $t\in[0,T]$. 

We will work on system \eqref{simeqn} and accomplish the proof of Theorem \ref{thmplc1} in three steps. 
First, for any fixed $J$, we solve for $\theta=\theta^J$ of the wave equation  $\eqref{simeqn}_2$  with $v_t$ replaced by $J$. Then, in the second step, substituting $\theta^J$ into  $\eqref{simeqn}_1$ and solving  for $v=v^{J}$, we obtain a map from $J$ to $\mathcal{M}(J):=v^J_t$,  which is basically \eqref{hatu}-\eqref{hatvt}, and then show this map has a fixed point, for a small time interval. Finally, we prove the energy decay and extend the local existence result to a global one.

Since we will use the Schauder fixed point theorem, we cannot achieve uniqueness in this paper. As a consequence, the solutions obtained for  $t\in[0,T_1]$ and for $t\in[0,T_2]$ with $T_2>T_1$ might not agree with each other  for $t\in[0,T_1]$. Hence we cannot claim an existence result for $t\in [0,\infty)$.
The uniqueness   for the variational wave equation was proved in \cite{BCZ}. 
So this might be a very interesting technique issue which, hopefully,  could be conquered later, although the method in this paper fails to give a uniqueness.


\subsection{Existence result for the wave equation with any given $J$\label{sec6.1}}
In this subsection, we first prove the existence of a weak solution for
\beq\label{192}
\theta_{tt}+\theta_t=c(\theta)(c(\theta)\theta_{x})_x
-J(x,t)
\eeq
where $J(x,t)$ is given for any $(x,t)\in \R\times\R^+$, with
\beq\label{JbarT}
\|J\|_{C^\alpha\cap L^\infty\cap L^2(\Omega_{T})}=: \bar{J}(T)<\infty,
 \eeq
 where 
 \[
 \Omega_{T}=\{(x,t)\,|\,x\in\R,\;t\in[0,T]\}
 \]
 for any $T\geq 0$. We fix an arbitrary time $T>0$, and only prove the existence of solution in $t\in[0,T]$. The equation \eqref{192} comes from $\eqref{simeqn}_2$ with $v_t$ replaced by $J$. 
 
 In the rest of paper, we always assume that $\alpha$ is a constant such that
\[
0<\alpha<\frac{1}{4}.
\]

By the initial condition \eqref{tecinitial}, there exist $x_a$ and $x_b$, such that,
$\theta_1(x)$, $\theta'_0(x)$ and $J(x,0)=u'_0(x)+\theta_1(x)$ are small enough for $x\not\in( x_a, x_b)$, particularly,  if $\Omega^a$ and $\Omega^b$ are domains of dependence with bases $(-\infty, x_a]$ and $[x_b,\infty)$, respectively, then solutions over $\Omega^a$ and $\Omega^b$ will not blow up before time $T$.

In fact, because $J$ is in $C^\alpha\cap L^2\cap L^\infty$ with $J(\pm \infty,0)=0$, for any $\ve>0$, when $x_a$ and $x_b$ have sufficiently large magnitudes, $|J(x,t)|<\ve$ for any $x<x_a$ or $x>x_b$ and $0\leq t\leq T$. Hence, one can find a priori bounds on gradient variables $R$ and $S$ using equations \eqref{S_sing}-\eqref{R_sing} and the initial smallness of $R$ and $S$. The existence and uniqueness for $C^1$ solutions in domains of dependence $\Omega^a$ and $\Omega^b$ is a classical result (see \cite{Lidaqian, liyu} together with equations \eqref{S_sing}-\eqref{R_sing}). One can also use the semilinear equation method to find this unique solution. Note, when $T$ increases, $|x_a|$ and $|x_b|$ may increase.

The solution of \eqref{192} has a finite speed of propagation, so we can split the region $\Omega_T$ into different domains of dependence. As described in Figure \ref{fig_last} and its caption, now we only have to consider a characteristic triangle $\Omega^0$, which, together with $\Omega^a$ and $\Omega^b$, covers $\Omega_T$.


Due to the assumption (\ref{condc}) that the wave speed $c$ has positive lower and upper bounds,   a time $T_0$, whose definition is clear from  Figure \ref{fig_last}, is finite.
 
\begin{figure}[htb]
\centering
\includegraphics[scale=.35]{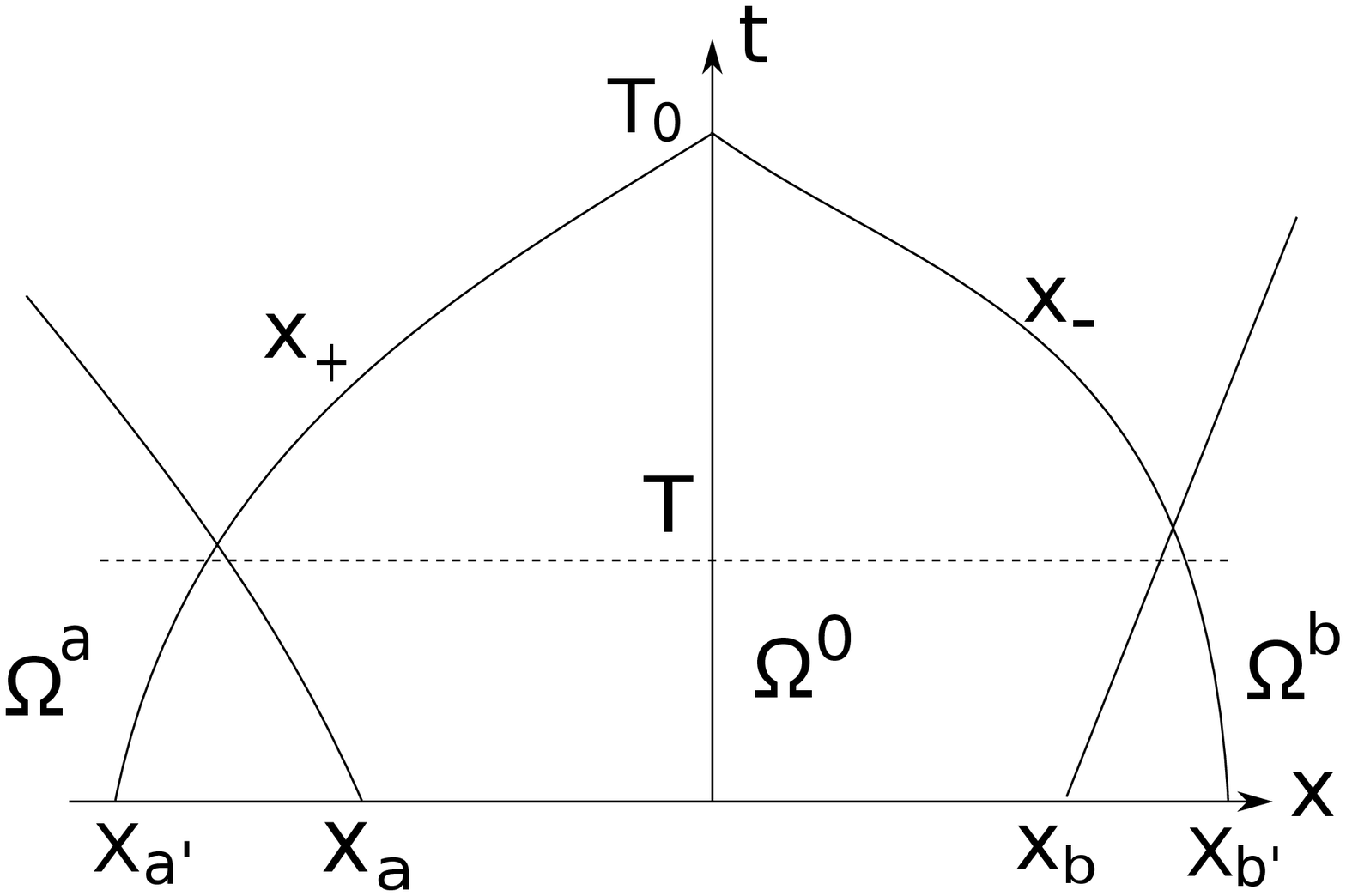}\qquad
\includegraphics[width=.35\textwidth]{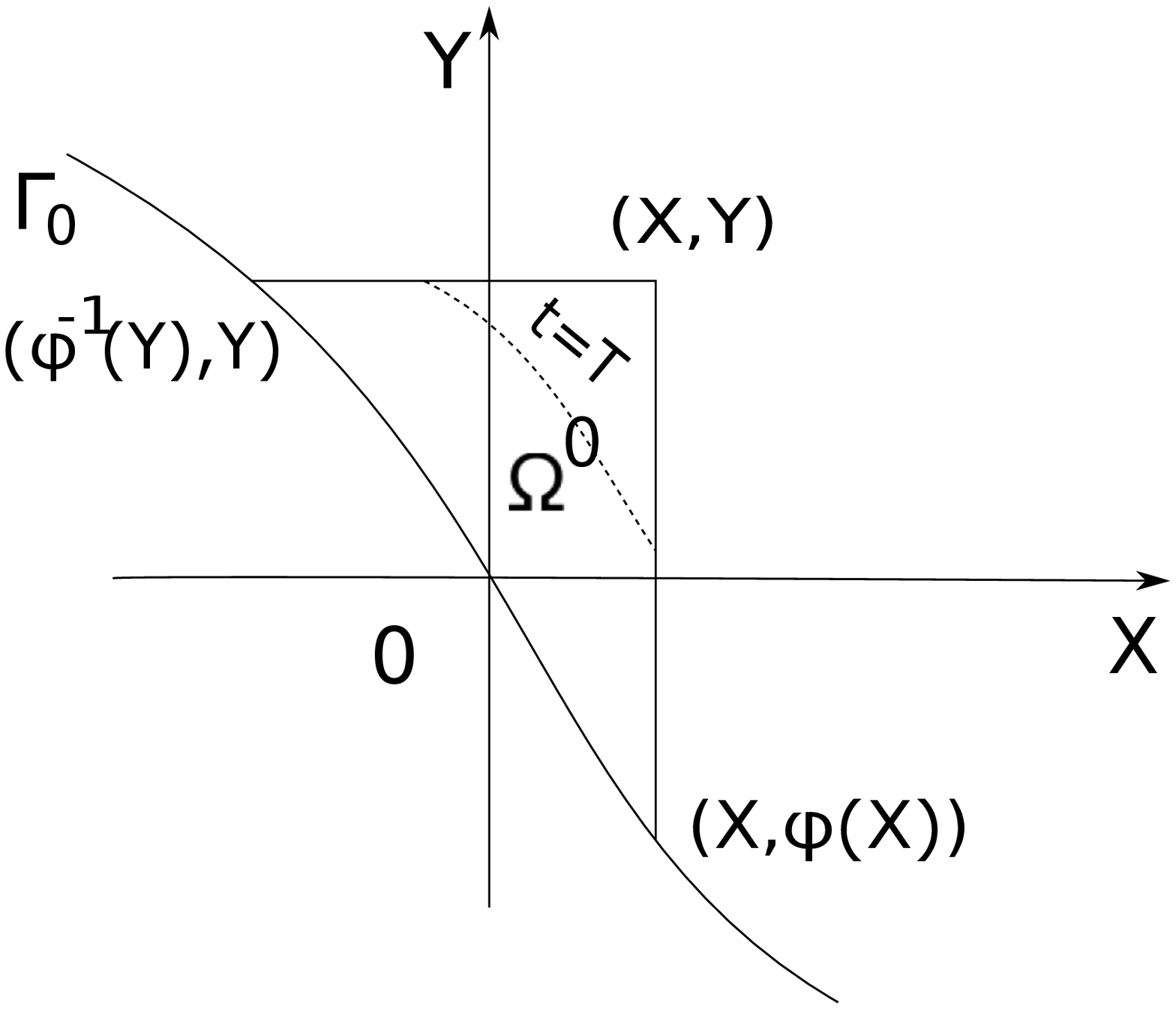}
\caption{{\em Left: $\Omega^a$ and $\Omega^b$ are domains of dependence corresponding to 
$x<x_a$ and $x>x_b$, respectively. When $-x_a$ and $x_b$ are sufficiently large, there will be no gradient blowup in $\Omega^a$ and $\Omega^b$  before $t=T$. $\Omega^0$ is another characteristic triangle, which covers $\Omega_T$ together with $\Omega^a$ and $\Omega^b$.
Right: region $\Omega^0$ on the $(X,Y)$-plane.}}
\label{fig_last}
\end{figure}

\subsubsection{Setup of the boundary value problem over $\Omega^0$ in the $(X,Y)$-coordinates.}
Now we only have to consider the region $\Omega^0$ in Figure \ref{fig_last}.

The key idea is to first construct the solution of semilinear system \eqref{semilinear_new} then change it back to the original system. This idea was used for the variational wave equation in \cite{BZ}.
The appearance of the source term $J$ in our equation makes the existence proof harder than that in \cite{BZ}. In fact, we do not have uniform a priori estimates for $p$ and $q$ by directly using the semi-linear system \eqref{semilinear_new}  in the $(X,Y)$-plane. Instead, we first establish the local existence for solutions in the $(X,Y)$-coordinates then transform it to the $(x,t)$-coordinates. Next, we will prove the key Lemma \ref{lemma3.3}, in which the $L^\infty$ estimates for $p$ and $q$ are given. This helps extend the solution to $t\in[0,T]$.




%

Now we start from the  boundary value problem on $\Omega^0$ in the $(X,Y)$-plane. The system for this boundary value problem is given in \eqref{192} with $J(x,t)$ given and satisfying \eqref{JbarT}.

The initial line $t=0$ in the $(x,t)$-plane is transformed to a
parametric curve 
\beq\label{i1} 
\Gamma_0 :=   \big\{ (X,Y):\,~~Y=\varphi (X)\big\}\subset {\mathbb
R}^2
\eeq
in the $(X, Y)$-plane,  where $Y = \varphi( X )$ if and only if
there is   $x$ such that 
\beq\label{12} 
 X =\displaystyle \int_0^x [1+R^2(y,0)]\,dy\;\mbox{ and }\;   
 Y =\ds \int_x^0[1+S^2(y,0)]\,dy.
  \eeq 
The curve $\Gamma_0$ is non-characteristic. 
The two functions $X= X(x), Y = Y(x)$  are well-defined
and absolutely continuous.
So
$\varphi (X)$ is continuous and strictly decreasing in $X$ since $X(x)$ is strictly
increasing while $Y(x)$ is strictly decreasing.  From \eqref{initial},
\eqref{E.00} and \eqref{S_R_def}, it follows 
\beq \big|X+\varphi (X)\big|  \leq 4 {\mathcal
E}(0)<\infty\,.\label{xp} \eeq 
As $(x,t)$ ranges over the domain
$\Omega^0$, the corresponding variable $(X,
Y)$ ranges over the set enclosed by one vertical line, one horizontal line and $\Gamma_0$, see Figure \ref{fig_last}. Without loss of generality, we still use $\Omega^0$ to denote this region in the $(X,
Y)$-plane.

Along the initial curve $\Gamma_0$ in \eqref{i1}
parametrized by $x\mapsto \big(X(x), \,Y(x)\big)$ using \eqref{12}, we can thus
assign the boundary data $\bar \theta, \bar z,\bar w$ $\in
L^\infty$ by their definitions in \eqref{defwz} evaluated at the initial data, and $\bar p=1$ and $\bar q=1$ by \eqref{defpq}, where we also used \eqref{initial}.
Therefore, it is easy to check that
\beq\notag
\int_{\Gamma_0}\frac{1-\cos \bar w}{4}\,dX-\frac{1-\cos \bar z}{4}\,dY\leq 2\mathcal E (0).
\eeq

Finally, we denote
 \beq\label{defD}
 D=  \hbox{Distance between two vertices of }\Omega^0 \hbox{ on }\Gamma_0.
 \eeq
Since the wave speed $c$ has positive lower and upper bounds, clearly, $D$ is a constant depending only on the initial condition.

\subsubsection{Local existence of the boundary value problem in $(X,Y)$ coordinates.}\label{LocE}
Now we first show a local existence result, by finding a fixed point of the map
\[
(\hat\theta, \hat w, \hat z, \hat p, \hat q)=\mathcal T_1(\theta, w, z, p,q)
\]
 defined by the solution of  \eqref{semilinear_new}, more precisely, 
\beq\label{hattheta}
\hat{\theta}(X,Y)=\theta(X,\phi(X))+\int_{\phi(X)}^Y\frac{\sin z}{4c}q (X,\bar Y)d\bar Y,
\eeq
\beq\label{hatz}
\begin{split}
\hat{z}(X,Y)=&z(\phi^{-1}(Y),Y)+
\int_{\phi^{-1}(Y)}^X p\Big\{\frac{c'}{4c^2}(\cos^2\frac{w}{2}-\cos^2\frac{z}{2})\\
&-\frac{1}{4c}(\sin w\cos^2\frac{z}{2}+\sin z\cos^2\frac{w}{2})
-\frac{1}{c}J(x_p,t_p)\cos^2\frac{z}{2}\cos^2\frac{w}{2}\Big\}(\bar X,Y)d\bar X ,
\end{split}
\eeq
\beq\label{hatw}
\begin{split}
\hat{w}(X,Y)=&w(X,\phi(X))+\int_{\phi(X)}^Y
q
\Big\{\frac{c'}{4c^2}(\cos^2\frac{z}{2}-\cos^2\frac{w}{2})\\
&-\frac{1}{4c}(\sin w\cos^2\frac{z}{2}+\sin z\cos^2\frac{w}{2})
-\frac{1}{c}J(x_m,t_m)\cos^2\frac{z}{2}\cos^2\frac{w}{2}\Big\}(X,\bar Y)d\bar  Y ,
\end{split}
\eeq
\beq\label{hatp}
\begin{split}
\hat{p}(X,Y)=&p(X,\phi(X))+\int_{\phi(X)}^Y pq\left\{\frac{c'}{8c^2}(\sin z-\sin w)\right.\\
&\left.
-\frac{1}{2c}[\frac{1}{4}\sin w\sin z+\sin^2 \frac{w}{2}\cos^2\frac{z}{2}]-\frac{1}{2c}J(x_m,t_m)\sin w\cos^2\frac{z}{2}\right\}(X,\bar Y)d\bar Y,
\end{split}
\eeq
\beq\label{hatq}
\begin{split}
\hat{q}(X,Y)=&q(\phi^{-1}(Y),Y)+\int_{\phi^{-1}(Y)}^X pq\left\{\frac{c'}{8c^2}(\sin w-\sin z)\right.\\
&\left.
-\frac{1}{2c}[\frac{1}{4}\sin w\sin z+\sin^2 \frac{z}{2}\cos^2\frac{w}{2}]-\frac{1}{2c}J(x_p,t_p)\sin z\cos^2\frac{w}{2}\right\}(\bar X,Y)d\bar X,
\end{split}
\eeq
where
\beq\label{xpdef}
x_p(\bar X,Y)=x(\bar X,\phi(\bar X))+\int_{\phi(\bar X)}^Y-\frac{1+ \cos z}{4}q d\tilde Y,
\eeq
\beq\label{tpdef}
t_p(\bar X,Y)=t(\bar X,\phi(\bar X))+\int_{\phi(\bar X)}^Y\frac{1+ \cos z}{4c}q d\tilde Y,
\eeq
\beq\label{xmdef}
x_m(X,\bar Y)=x(\phi^{-1}(\bar Y),\bar Y)+\int_{\phi^{-1} (\bar Y)}^X-\frac{1+ \cos w}{4}p d\tilde X,
\eeq
and
\beq\label{tmdef}
t_m(X,\bar Y)=t(\phi^{-1}(\bar Y),\bar Y)+\int_{\phi^{-1} (\bar Y)}^X-\frac{1+ \cos w}{4c}p d\tilde X.
\eeq

Note that $(x_p,t_p)$
and $(x_m,t_m)$ come from different but equivalent equations in \eqref{eqnxt}. More
importantly, such a choice makes sure that $(x_p,t_p)(\bar X, Y)$ and $(x_m,t_m)(X, \bar Y)$
have finite partial derivative in $Y$ and $X$, respectively.

For brief, we set 
\[
V=(\theta, w, z, p,q)\quad \hbox{and}\quad \hat V=(\hat\theta, \hat w, \hat z, \hat p, \hat q),
\]
and let 
\[\bar{V}(X)=(\theta, w, z, p,q)(X,\phi(X))\]
 denote the initial data.

Fix a constant $K_1$ large  to be determined later on (say $K_1\gg 2\|\bar V(X)\|_{C^{\alpha}\cap L^{\infty}}$). Define the set
\beq\label{defK1}
\mathcal{K}_1=\left\{V\,\big |\,\| V(X,Y)\|_{C^{\alpha}\cap L^\infty(\hat \Omega_{\delta})}\leq K_1
,\ \ V(X, \phi(X))=\bar V(X)\right\} 
\eeq
where $\delta>0$ is a constant and 
\beq\label{Ome_del}
\hat \Omega_{\delta}=\big\{(X,Y)\in \Omega^0:\; \mbox{dist}((X,Y), \Gamma_0)\leq\delta \big\}.
\eeq

Note that the space $\mathcal{K}_1$ is compact in $C^0(\hat \Omega_{\delta})$, where $\hat \Omega_{\delta}$ is a two dimensional bounded connected region.
We will apply the Schauder Fixed-Point Theorem to show that, there exists $\delta>0$, such that, the map $\mathcal{T}_1$ has a fixed point in the set $\mathcal{K}_1$.


The proof is standard and we will check the conditions for the theorem  briefly. 

For any local solution, when $\delta$ is small enough, it is easy to find uniform a priori bounds on $p$ and $q$ by corresponding equations in \eqref{semilinear_new}. We omit details here because we will later give much stronger global a priori bounds on $p$ and $q$, which will allow us to extend the solution to a global one.

Since $\mathcal{K}_1$ is a compact set in $C^0(\hat \Omega_{\delta})$ space, it suffices to show that the map $\mathcal{T}_1$ is continuous under  $C^0$ norm and maps $\mathcal{K}_1$ to itself.

It follows directly from \eqref{hattheta}-\eqref{hatq} that the map $\mathcal T_1$ is continuous since $J(x,t)$ is continuous in $x$ and $t$, and that  $\mathcal T_1(V)(X,\phi(X))=V(X,\phi(X))$.
Now recall that $J(x,t)$  is H\"older continuous in $x$ and $t$, and $(x_p,t_p)(\bar X, Y)$ and $(x_m,t_m)(X, \bar Y)$ have finite partial derivatives with respect to $Y$ and $X$. By
\eqref{hattheta}-\eqref{hatq}, if $\delta$ is small enough,  $\mathcal{T}_1$ maps $\mathcal{K}_1$ to itself. For example, $\hat p_Y$ is bounded, so $\hat p$ is Lipschitz in the $Y$ direction. In the $X$ direction, the $C^\alpha$ norm of $\hat p$ is different from the $C^\alpha$ norm of $p(X,\phi(X))$ by $C\delta$ times $\|V\|_{C^\alpha\cap L^\infty}$ for some constant $C$, where we use $J(x_m,t_m)$ is H\"older in $(x_m,t_m)$, and $(x_m,t_m)$ is Lipschitz in the $X$ direction.

So, we claim that there is a fixed point for the map $\mathcal{T}_1$ in $\mathcal{K}_1$.

 Note that $\delta$ depends on $K_1$ while $K_1$ can be determined by the a priori bound on $\| V(X,Y)\|_{C^{\alpha}\cap L^\infty(\Omega^0)}$. Later, in Lemma \ref{lemma3.3}, we provide the a priori bounds on $p$ and $q$ in $\Omega^0$, by the map ${\mathcal T}_1$, and then we can find the $C^{\alpha}$ bound on $V$ as we just discussed. Clearly, the $L^\infty$ bound on $V$ is finite. So $K_1$ will be given by a constant only depending on the initial data $\bar{V}$.  As a consequence, $\delta$ is fixed with respect to $K_1$. This allows one to apply the same procedure to extend the solution to $\Omega^0$.


\begin{remark}
The equations of $\theta_X$ and $\theta_Y$ are equivalent since it is easy to check that $\theta_{XY}=\theta_{YX}$, as in \cite{BZ}. The semilinear system \eqref{semilinear_new} are invariant under translation by $2\pi$ in $w$ and $z$. It would be more precise to work with the variables $e^{iw}$ and $e^{iz}$. For simplicity, we shall use the variables $w$ and $z$ keeping in mind that they range on the unit circle $[-\pi,\pi] $ with endpoints identified. 
\end{remark}

\subsubsection{Inverse transformation\label{Sec_inv}}
Now we can carry out the inverse transformation from the $(X,Y)$-coordinate to the $(x,t)$-coordinate. This step is very similar to the one in \cite{BZ}. Let's only briefly introduce the key steps for completeness.

For convenience, we recall \eqref{eqnxt} below
\beq\label{eqnxt1}
\left\{
\begin{array}{ll}
\ds x_X=\frac{1}{2X_x}=\frac{1+\cos w}{4}p,\\
\\
\ds x_Y=\frac{1}{2Y_x}=-\frac{1+\cos z}{4} q,
\end{array}
\right.
\qquad 
\left\{
\begin{array}{ll}
\ds t_X=\frac{1}{2cX_x}=\frac{1+\cos w}{4c}p,\\
\\
\ds t_Y=-\frac{1}{2cY_x}=\frac{1+\cos z}{4c} q.
\end{array}
\right.
\eeq

It is easy to check that 
\[
t_{XY}=t_{YX},\qquad x_{XY}=x_{YX}
\]
using \eqref{eqnxt1} and \eqref{semilinear_new}, so two $t$ equations and two $x$ equations in \eqref{eqnxt1} are equivalent, respectively.
Hence, by \eqref{xpdef}-\eqref{tmdef},
\[
(x_m,t_m)=(x_p,t_p)=:(x,t)
\]
And, \eqref{eqnxt1} provides an inverse transformation from the $(X,Y)$-coordinates
to the $(x,t)$-coordinates.

Next,  we note that   the map from $(X,Y)$ to $(x,t)$ may  not be one-to-one, since $(t_X,x_X)$ and $(t_Y,x_Y)$ might vanish as singularity forms. But,  
if 
\[x(X_1,Y_1)=x(X_2,Y_2),\quad t(X_1,Y_1)=t(X_2,Y_2),\] 
then 
\[
\theta(x(X_1,Y_1),t(X_1,Y_1))=\theta(x(X_2,Y_2),t(X_2,Y_2)),
\]
and hence, the function $\theta(x,t)$ is well defined and
\beq\label{xt_co}
dx\,dt=\frac{pq}{2c(1+R^2)(1+S^2)}dX\,dY=\frac{pq}{2c}\cos^2\frac{w}{2}\cos^2\frac{z}{2}\,dX\,dY.
\eeq
We omit the proof here and refer the readers to \cite{BZ} for more details.

\subsubsection{Global existence of \eqref{semilinear_new}} 
Now we extend the local solution of \eqref{semilinear_new} established in Section \ref{LocE} to $\Omega_T$.
{ As we discussed, we only need some global a priori bounds on $p$ and $q$, which will be given in the next key lemma. }


\begin{lemma}\label{lemma3.3} Consider any solution of \eqref{semilinear_new} constructed in the local existence result up to the region $\Omega^0$. Then, we have 
\beq\label{bdspq}
0<A_1\leq \max_{(X,Y)\in \Omega^0 }\left\{p(X,Y),\, q(X,Y)\right\}\leq A_2,
\eeq
for some constants $A_1$ and $A_2$ specified in the proof.
\end{lemma}
\begin{proof}
	\begin{figure}[htp] \centering
		\includegraphics[width=.5\textwidth]{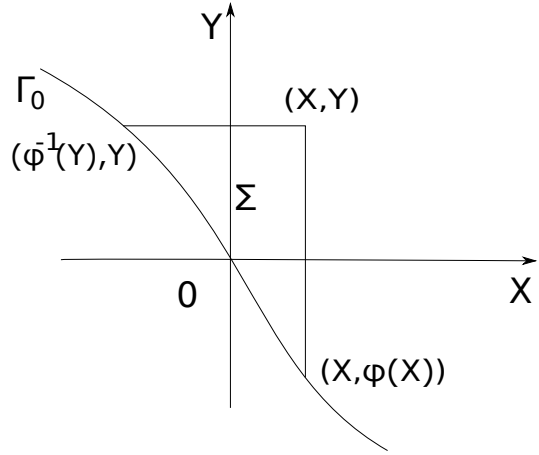}
	\caption{The domain $\Sigma$.}\label{fig:sig}
	\end{figure}
We consider a characteristic triangle $\Sigma$ in $\Omega^0$ enclosed by the line segment between $(X,Y)$ and $(X,\varphi(X))$, the line segment between $(X,Y)$ and $(\varphi^{-1}(Y),Y)$, and $\Gamma_0$. See Figure \ref{fig:sig}.
 If we denote equations of $p$ and $q$ in \eqref{semilinear_new} as
\[
q_X=A pq,\quad p_Y=B pq,
\]
then $p$ and $q$ are positive since $p=q=1$ on $\Gamma_0$. Furthermore,  direct computation using \eqref{semilinear_new} gives
\beq\notag
\begin{split}
\int_{\partial \Sigma}p\,d\bar X-q\,d\bar Y
=&-\iint_{\Sigma}q_X+p_Y\,d\bar Xd\bar Y\\
=&\iint_{\Sigma}\frac{pq}{2c}(\sin \frac{w}{2}\cos \frac{z}{2}+\sin \frac{z}{2}\cos \frac{w}{2})^2\,d\bar Xd\bar Y\\
&+\iint_{\Sigma}\frac{pq}{2c}J (\sin z\cos^2\frac{w}{2}+\sin w\cos^2\frac{z}{2})\,d\bar Xd\bar Y.
\end{split}
\eeq
 Thus, using $p=q=1$ on $\Gamma_0$ again,
\begin{eqnarray}
&&\int_{\varphi^{-1}(Y)}^Xp(\bar X, Y)\,d\bar X+\int^{Y}_{\varphi(X)}q(X, \bar Y)\,d\bar Y\nn\\
&&\quad =X-\varphi^{-1}(Y)+Y-\varphi(X)-\iint_{\Sigma}\frac{pq}{2c}(\sin \frac{w}{2}\cos \frac{z}{2}+\sin \frac{z}{2}\cos \frac{w}{2})^2\,dXdY\nn\\
&&\qquad-\iint_{\Sigma}\frac{pq}{2c}J (\sin z\cos^2\frac{w}{2}+\sin w\cos^2\frac{z}{2})\,dXdY \label{intqp}\\
&&\quad=X-\varphi^{-1}(Y)+Y-\varphi(X)-\iint_{\Sigma}\frac{pq}{2c}(\tan \frac{w}{2}+\tan \frac{z}{2})^2 \cos^2 \frac{z}{2}\cos^2\frac{w}{2}\,dXdY\nn\\
&&\qquad-\iint_{\Sigma}\frac{pq}{2c} 2 J (\tan \frac{w}{2}+\tan \frac{z}{2}) \cos^2 \frac{z}{2}\cos^2\frac{w}{2}\,dXdY\nn\\
&&\quad \leq X-\varphi^{-1}(Y)+Y-\varphi(X)
+\iint_{\Sigma}\frac{pq}{2c} |J|^2 \cos^2 \frac{z}{2}\cos^2\frac{w}{2}\,dXdY\nn\\
&&\quad\leq  X-\varphi^{-1}(Y)+Y-\varphi(X)+\iint_{\tilde\Sigma}|J|^2\,dxdt\nn\\
&&\quad\leq  2D+\bar{J}(T)
\end{eqnarray}
where we use \eqref{JbarT}, \eqref{defD} and \eqref{xt_co} in the last step. The constant $\bar J(T)$ is defined in \eqref{JbarT}.
Here $\tilde\Sigma$ denotes the characteristic triangle in the $(x,t)$ plane transformed from $\Sigma$.


 It then follows from $ p_Y=B pq$ that
\[
p(X,Y)\leq e^{\sup{|B|} \, \int_{\phi(X)}^Y q(X,Y')\,dY'}<e^{\sup{|B|}\, (2D+\bar{J}(T))},
\]
where $\sup{|B|}$ also depending on $\bar{J}(T)$.

The other bounds in \eqref{bdspq} can be found similarly.
\end{proof}

\subsubsection{Global existence for the wave equation\label{Sec_wave}}
Finally,  we show that $\theta(x,t)$ is a weak solution for the wave equation \eqref{192} in the $(x,t)$-plane for $0\leq t\leq T$, i.e. on $\Omega_T$.

We can prove the local H\"older continuity of $\theta$ in both $x$ and $t$ with exponent $1/2$ by showing that the integrals of $(\theta_t+c(\theta)\theta_x)^2$ and $(\theta_t-c(\theta)\theta_x)^2$  along forward and backward characteristics, respectively, are bounded. 
Then using the Sobolev embedding from $H^1$ to $C^{1/2}$. Note we also use the property on wave speed $c$ in \eqref{condc}. This also shows that all characteristic curves are $C^1$ with H\"older continuous derivative. 
Finally,  
one can show that
\beq\label{weak2.0}
\iint_{\Omega_0} \left(\theta_t \phi_t-(c(\theta)\phi)_x c(\theta)\theta_x -\theta_t \phi-J\phi\right)\,dxdt=0
\eeq
for any test function $\phi$ on $\Omega^0$. Applying this on the overlap domain  $\Omega^a\cap \Omega^0$ or $\Omega^b\cap \Omega^0$, and by the uniqueness of $C^1$ solution on  $\Omega^a$ and $\Omega^b$, we can  glue solutions found on three regions $\Omega^a$, $\Omega^b$ and $\Omega^0$ to get a solution $\theta(x,t)$ for \eqref{192}, when $t\in[0,T]$, in the weak sense of
\beq\label{weak2}
\int_0^{T}\int_{\mathbb R} \left(\theta_t \phi_t-(c(\theta)\phi)_x c(\theta)\theta_x -\theta_t \phi-J\phi\right)\,dxdt=0
\eeq
for $\phi\in C_0^\infty\{(x,t)\in\mathbb R\times [0,T]\}$.
Since the proofs of these results are very similar to those in \cite{BZ}, we   refer the readers to \cite{BZ} for details.

Finally, for this part,  we will prove that
\beq\label{6.20}
E(t)=\int_{-\infty}^\infty (\theta_t^2+c^2(\theta)\theta_x^2)(x,t)dx<C_E
\eeq
for a constant $C_E$ depending on $E(0)$ and $J$, where $E(t)$ is first given in \eqref{E_def}.

For any $\tau\geq0$, let $\Gamma_\tau\subset\Omega^+$ be the transformation of the horizontal line  $t=\tau$ in the $(x,t)$-plane. 
\begin{figure}[htp] \centering
		\includegraphics[width=.35\textwidth]{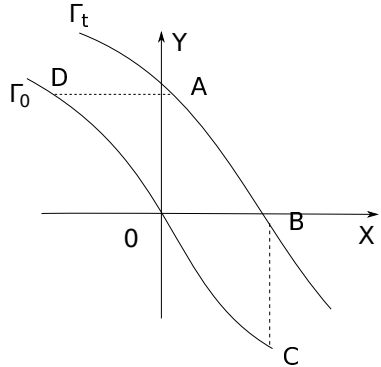}
	\caption{The domain $D_t$ is enclosed by four curves with vertices $A,B,C,D$.}\label{fig:fig3}
	\end{figure}
We first consider the bounded domain $D_{t}$ in the $(X,Y)$-plane in Figure \ref{fig:fig3},
where $D_{t}$ is enclosed by two curves $AB$ and $CD$ with
\beq\label{last1}
AB\subset\Gamma_t\cap\partial D_t,\quad CD\subset\Gamma_0\cap\partial D_t\eeq
and two straight line-segments $BC$ and $DA$. Here we denote four vertices in $(x,t)$ coordinates as 
\beq\label{last2}
A=(a,t),\quad B=(b,t),\quad C=(c,0),\quad D=(d,0)\eeq
for some $a<b$ and $d<c$, respectively. 

By Green's theorem, 
\beq
\label{eqn3.2}
\begin{split}
&\int_{\partial D_t}\frac{1-\cos w}{4}p\,dX-\frac{1-\cos z}{4}q\,dY\\
 &\qquad =-\frac{1}{4}\iint_{D_t}\big[((1-\cos z)q)_X+((1-\cos w)p)_Y\big]\,dXdY.
\end{split}
\eeq
Direct computation from \eqref{semilinear_new} gives
\beq\label{pplusq}
\begin{split}
&((1-\cos z)q)_X+((1-\cos w)p)_Y\\
&\quad =-\frac{pq}{c}(\sin \frac{w}{2}\cos \frac{z}{2}+\sin \frac{z}{2}\cos \frac{w}{2})^2-\frac{pq}{c}J (\sin z\cos^2\frac{w}{2}+\sin w\cos^2\frac{z}{2}).
\end{split}
\eeq
Substituting it into \eqref{eqn3.2}, and using the transformation relation in \eqref{eqnxt}, the following inequality holds, where curves $AB$ and $DC$ have starting points $A$, $D$ and ending points $B$, $C$, respectively.
\beq\label{eqn3.3}
\begin{split}
&\int_a^b (\theta_t^2+c^2(\theta)\theta_x^2)(x,t)\,dx\\
&\quad =\int_{AB\cap (\cos w\neq -1)}\frac{1-\cos w}{4}p\,dX
+\int_{BA\cap (\cos z\neq -1)}\frac{1-\cos z}{4}q\,dY\\
&\quad \leq\int_{AB}\frac{1-\cos w}{4}p\,dX-\frac{1-\cos z}{4}q\,dY\\
&\quad =\int_{DC}\frac{1-\cos w}{4}p\,dX-\frac{1-\cos z}{4}q\,dY-\int_{DA}\frac{1-\cos w}{4}p\,dX\\ 
&\qquad-\int_{CB}\frac{1-\cos z}{4}q\,dY -\frac{1}{4}\iint_{D_t}\frac{pq}{c}(\sin \frac{w}{2}\cos \frac{z}{2}+\sin \frac{z}{2}\cos \frac{w}{2})^2\,dXdY\\
&\qquad -\frac{1}{4}\iint_{D_t}\frac{pq}{c}J (\sin z\cos^2\frac{w}{2}+\sin w\cos^2\frac{z}{2})\,dXdY\\
&\quad \leq  \int_{DC}\frac{1-\cos w}{4}p\,dX-\frac{1-\cos z}{4}q\,dY-2\iint_{\mathcal{D}}\theta^2_t\,dxdt-2\iint_{\mathcal{D}}J\theta_t\,dxdt\\
&\quad=\int_d^c (\theta_t^2+c^2(\theta)\theta_x^2)(x,0)\,dx-2\iint_{\mathcal{D}}\theta^2_t\,dxdt-2\iint_{\mathcal{D}}J\theta_t\,dxdt,
\end{split}
\eeq
where we have used the following fact in the second last step: 
\beq\notag
\left|\frac{\partial(X,Y)}{\partial(x,t)}\right|
=\left|
\begin{array}{cc}
\displaystyle X_x& \displaystyle X_t\\
\displaystyle Y_x& \displaystyle Y_t
\end{array}
\right|=-2cX_xY_x=\frac{8}{pq}\frac{1}{1+\cos w}\frac{1}{1+\cos z},
\eeq
and $\mathcal{D}$ is the region in the $(x,t)$-plane transformed from $D_t$. The last equal sign holds since there exists no energy concentration initially using $\theta_0(x)$ is absolutely continuous.
\medskip

If we let $(a, b)$ goes to $(-\infty,\infty)$, then by  \eqref{eqn3.3} we have, for any $0\leq t
\leq T$,
\[
E(t)\leq E(0)+4  \int_0^t\int_{-\infty}^\infty|J||\theta_t|\,dxdt,
\]
and hence, 
\[
\frac{1}{2}\max_{0\leq t
\leq T}E(t)\leq E(0)+4C_\ve\int_0^T\int_{-\infty}^\infty|J|^2\,dxdt
\]
for some constant $C_\ve$ and $C$. This gives \eqref{6.20}, and implies that $\theta_t(\cdot, t)$ and $\theta_x(\cdot, t)$ are both square integrable functions in $x$, so do $R$ and $S$.

%

In summarize, we have 

\begin{lemma}\label{lemma6.2}
 Fix any $T>0$ and any function $J(x,t)$ satisfying \eqref{JbarT} over $\Omega_T$.
If the same  assumptions on initial conditions in Theorem \ref{thmplc1} hold, then 
there exists a weak solution of \eqref{192} over $\Omega_T$ with bounded energy $E(t)<C_E$ for some $C_E$ depending on $E(0)$ and $J$.
\end{lemma}

\subsection{A map from $J(x,t)$ to $J(X,Y)$}
Using similar method as in  Lemmas 1 and 2 in \cite{BCZ} for variational wave equation (i.e. \eqref{192} with $J\equiv 0$), 
we can show that the uniqueness of forward and backward characteristics for equation \eqref{192}, i.e. the uniqueness of the $(X,Y)$ coordinates. Hence  $(x_m(X,Y),t_m(X,Y))=(x_p(X,Y),t_p(X,Y))$ is unique. As a consequence, for any given $J(x,t)\in C^\alpha\cap L^2\cap L^\infty$, solution we found previously satisfies the semilinear system \eqref{semilinear_new} with a unique source term $\tilde J(X,Y)= J(x(X,Y),t(X,Y))$. 

Clearly $\tilde J(X,Y)$ is $L^\infty$. In fact, for any given $\tilde J(X,Y)\in L^\infty$, the semilinear system \eqref{semilinear_new} has a unique solution obtained as the fixed point of a contract mapping in a weighted $L^\infty$ space, in the same way as the proof in \cite{BC}. 
As a consequence, we know that  the solution $(x,t,\theta, w, z, p,q)(X,Y)$ of  the semilinear system \eqref{semilinear_new} is continuously dependent on $J$ in the $L^\infty$ space, due to the uniform bounds on $p$ and $q$ and the relation \eqref{eqnxt}. 
The proof is straightforward and  we omit the detail.
%

\begin{remark}
The main idea in \cite{BCZ} for the variational wave equation is to find some  Lipschitz weighted distance between any two characteristics. This distance essentially measures the forward and backward energy between two characteristics. The Lipschitz property protects the uniqueness of characteristics. By adding $J(x,t)\in C^\alpha\cap L^2\cap L^\infty$, there will be some  new lower order terms $-2SJ$ and $-2RJ$ in the energy balance laws \eqref{A8}. This will not cause any problem since $S,R,J$ are all in $L^2$. The appearance of damping term never makes trouble here. Since this uniqueness result can be shown in a very similar way as in \cite{BCZ}, we omit the proof here.
\end{remark}
\begin{remark}
The $(X,Y)$-coordinates system is ideal for the wave equation but it is not the case for the heat equation. The original $(x,t)$-coordinates system works better for the heat equation. 
So we run our proof mainly in the $(x,t)$-coordinates system and use the continuous dependence of solution on $J$ in the  $(X,Y)$-coordinates system.

From $J(x,t)$ to $\tilde J(X,Y)$, there is actually an unfolding process when singularity forms, i.e. when characteristics tangentially touch each other.
\end{remark}
%

\subsection{A fixed point argument for a map on $J$\label{sec_6.2}}
Recall that, for any given function $J(x,t)\in (C^\alpha\cap L^\infty\cap L^2)(\Omega_T)$, there exists a weak solution  $\theta(x,t)=\theta^J(x,t)$ of \eqref{192}.

Using relations \eqref{hatvt} and \eqref{hatu}, we define a function $\mathcal M(J)$ on $\Omega_T$ as follows
\beq\label{hatvtL}
\begin{split}
\mathcal M(J)(x,t)
=&\int_{\R}H(x-y,t)\left(u_0'(y)+\theta_1(y)\right)\,dy\\
&+\int_0^{t}\int_{\mathbb R}H(x-y,t-s)\big[2\theta_s+c'(\theta)c(\theta)(\theta_y)^2\big](y,s)\,dyds\\
&-\int_0^{t}\int_{\mathbb R}H_y(x-y,t-s)\big[c^2(\theta)\theta_y-u\big](y,s)\,dyds
\end{split}
\eeq
for $t\in[0,T]$ and
\begin{align}\label{hatuL}
\begin{split}
u(x,t) =&\int_{\R}H(x-y,t)u_0(y)\,dy+\int_0^{t}\int_{\R}H_x(x-y,t-s)\theta_{s}(y,s)\,dyds.
\end{split}
\end{align}
This gives a map 
\beq\label{Lmap}
\mathcal T:\ J(x,t)\rightarrow \mathcal M(J)(x,t)
\eeq
on $C^{\alpha}\cap L^\infty\cap L^2(\Omega_T)$.

Initially, we define a set $\mathcal{K}$ as follows: for some constants $\delta$ and $K$,
\beq\label{defK}
\mathcal{K}=\left\{J(x,t)\,\big |\,\|J(x,t)-J^0(x,t)\|_{C^{\alpha}\cap L^\infty\cap L^2(\Omega_\delta)}\leq K, \ J(x,0)=\theta_1(x)+u'_0(x)
\right\}
\eeq
where 
\beq\notag
\begin{split}
J^0(x,t)=\int_{\R}H(x-y,t)\left(u_0'(y)+\theta_1(y)\right)\,dy
\end{split}
\eeq
 and
\[
\Omega_\delta=\{(x,t)\,|\,x\in\R,\;t\in[0,\delta]\}.
\]


{ 
The goal in this subsection is first to show, for a given large $K$, if $\delta$ is sufficiently small, then $\mathcal T$ has a fixed point on $\mathcal{K}$. 
In general, $\delta$ would depend  on $K$. Later,
after proving the energy estimate in Theorem \ref{lemma3.2}, we will establish  a uniform bound on ${\mathcal E}(t)$ and we will be able to choose $K$ as the a priori bound on $\|J(x,t)-J^0(x,t)\|_{C^{\alpha}\cap L^\infty\cap L^2(\Omega_\delta)}$, which  depends only on the initial data. As a consequence, $\delta>0$ can be fixed in terms of the initial data. Hence, one can extend the existence to $\Omega_T$ in finite step.}
\medskip

The existence of a fixed point for small $\delta$ will be accomplished in several steps. 
\smallskip

\paragraph{\bf (i).} { As before, we consider the far field separately, where the solution of forced wave equation is smooth due to finite propagation.

At time $t=0$, $J(x,0)=\theta_1(x)+u'_0(x)$ for any function $J$ in $\mathcal K$. Furthermore, all functions in $\mathcal K$
have uniform $C^\alpha$, $L^\infty$ and $L^2$ bounds. So by a very similar argument as in the beginning of Section \ref{sec6.1}, we can find a domain $\Omega^0$
as in Figure \ref{fig_last}, such that for any function $J$, the corresponding $\theta^J$, solved by the forced wave equation, has no singularity in the region $\Omega_T/\Omega^0$. Note that the solution $\theta^J$ solved by the forced wave equation is finitely propagating.

The choose of $\Omega^0$ depends on both $T$ and $K$. }
\smallskip

\paragraph{\bf (ii). }
To apply the Schauder Fixed-Point Theorem, we use the following Banach space $L^*$, which includes all functions $f$ in $\bigcap\limits_{p\in[2+\sigma,\infty)} L^p$ with a finite norm
\[
\|f\|_{L^*(\Omega_\delta)}=\sup\limits_{p\in[2+\sigma,\infty)}\|f\|_{L^p(\Omega_\delta)}< \infty.
\]

Here $\sigma$ is any given positive constant. It is easy to check that $L^*$ is a Banach space. Clearly, $\mathcal K\subset L^*$ and is bounded, since for any $f\in \mathcal K$,
\beq\label{holder}
\|f\|_{L^p}\leq \max(\|f\|_{L^2}, \|f\|_{L^\infty})<K, \quad p\in[2+\sigma,\infty).
\eeq
This can be easily proved by the H\"older inequality.

\medskip
\noindent {\em Claim: } $\mathcal K$ is a compact set in $L^*$.
 \medskip

We divide the proof of the claim into two steps. 

\smallskip
\noindent{\em Step 1: } We use Frechet-Kolmogorov theorem (\cite{HB,OH}) to prove that  
$\mathcal{K}$ is a compact subset of $L^p(\Omega_\delta)$ for any $p\in[2+\sigma,\infty)$. 
\smallskip

We can verify two conditions in Frechet-Kolmogorov theorem as following. For any $p\in[2+\sigma,\infty)$ and $f\in \mathcal K$,
\[
\begin{split}
&\Big(\int_{0}^\delta\int_{\R}|f(x+a, t+b)-f(x,t)|^p\, dxdt\Big)^{\frac{1}{p}}\\
&\leq
\|f(x+a, t+b)-f(x,t)\|_{L^\infty}^l(\int_{0}^\delta\int_{\R}|f(x+a, t+b)-f(x,t)|^{p(1-l)}\, dxdt)^{\frac{1}{p}} 
\end{split}
\]
where $p(1-l)>2$. Because $f$ has a uniformly bounded $C^{\alpha}\cap L^\infty\cap L^2$ norm, $\|f(x+a)-f(a)\|_{L^p(\Omega_\delta)}$ uniformly approaches zero as $a\rightarrow 0$ in $\mathcal K$.

By the same argument as in paragraphs {\bf (i)}, for any $\epsilon$, there exists a constant $r_0>0$ such that, if $r>r_0$, then $J(x,t)$
is less than $\epsilon$ for any $x>r$, $0\leq t\leq \delta$. So
\[
\int_{0}^\delta\int_{|x|>r}|f|^p\, dxdt\leq
\epsilon^{pl}\, \int_{0}^\delta\int_{\R}|f|^{p(1-l)}\, dxdt\leq \epsilon^{pl}(\max\{\|f\|_{L^2},\|f\|_{L^\infty}\})^{p(1-l)},
\]
with $p(1-l)>2$. This shows the uniform convergence of $\int_{0}^\delta\int_{|x|>r}|f|^p\, dxdt$ as $r$ goes to $\infty$.

It follows from Frechet-Kolmogorov theorem that
$\mathcal{K}$ is a compact subset in $L^p(\Omega_\delta)$ with $p\in[2+\sigma,\infty)$. 

\smallskip
\noindent{\em Step 2: } We show that $\mathcal{K}$  is a compact subset in $L^*$.
\smallskip

In fact, for any bounded set in $\mathcal{K}$, we can find a convergent sequence $f_n^1$ in $L^{2+\sigma}$. Then in $f_n^1$, we can find a convergent sub-sequence $f_n^2$ converging in $L^3$, then until any $L^n$. Now we claim $f_n^n$ converges in $L^*$. Clearly, all sub-sequences converge to the same function $f^*$ in a.e. sense. Furthermore, since $f_n^n\in \mathcal{K}$, one has 
$\|f_n^n\|_{L^p}<K$ for $p\in [2+\sigma,\infty)$ by \eqref{holder}. 
Therefore, $\|f^*\|_{L^p}<K$ for $p\in [2+\sigma,\infty)$, and hence,  $f^*\in L^*$ with $\|f^*\|_{L^*}<K$. The compactness proof is done, which completes the proof of the claim.

\medskip
\paragraph{\bf (iii).}

To use the Schauder Fixed-Point Theorem for $\mathcal T$ on $\mathcal{K}$, we  have to prove: 
\begin{itemize}
\item[(a).]
The map $\mathcal T$ maps from $\mathcal{K}$ to itself, for a small time interval $(0,\delta)$. 
\item[(b).] The map $\mathcal T$  is continuous under the $L^*$ norm.
\end{itemize}

\smallskip
\noindent{\em Proof of (a).}
First we show that $\|\mathcal M(J)(X,Y)\|_ {L^\infty\cap L^2(\Omega_\delta)}<K$. In fact, it is easy to show that 
the super norm of $u$ defined in \eqref{hatuL} is bounded. Then by the $L^2$ bound of $\theta_t$ in \eqref{6.20} and using a similar argument as in Lemma \ref{lemma3.1}, { we can show that 
\[\|\mathcal M(J)(x,t)-J^0(x,t)\|_ {L^\infty(\Omega_\delta)}<{C\cdot \delta^{\frac{1}{4}}}<K\] 
for $\delta$ small enough, where $C$ is a constant.}

Now we treat the $L^2$ norm of $\mathcal M(J)$.
By \eqref{hatuL}, it is easy to see that $u(x,t)$ is a weak solution of
\beq\label{key1}
u_{t}-u_{x x}=\theta_{t x}.
\eeq
Similarly, by \eqref{hatvtL}, $\mathcal M(J)(x,t)$ is a weak solution of
\beq\label{key2}
{\mathcal M}_{t}-{\mathcal M}_{x x}=c(\theta)(c(\theta)\theta_{x})_{x}
 -u_{x}-2\theta_{t}
\eeq
where the calculation is very similar to \eqref{Hdef}-\eqref{kexp}. 
Then by same arguments as in Lemma \ref{lemma3.1} and using \eqref{6.20}, { we know
\[\|\mathcal M(J)(x,t)-J^0(x,t)\|_ {L^2(\Omega_\delta)}<{C\cdot \delta^{\frac{1}{4}}}<K\] 
if $\delta$ is small}. Here, as before, we get use of the $t^{1/4}$ power in Lemma \ref{lemma3.1}.

{ Finally, if $\delta$ is small, the estimate 
\[
\|\mathcal M(J)(x,t)-J^0(x,t)\|_ {C^\alpha(\Omega_\delta)}<{C\cdot \delta^{\frac{1}{4}-\alpha}}<K
\] 
when  $\delta$ is small enough,
can be derived directly from Lemma \ref{lem2} in Appendix \ref{secA3}. }

\smallskip
\noindent{\em Proof of (b).}
First recall  $(\theta, z, w, p, q, x,t)(X,Y)$ is Lipschitz continuously dependent on $\tilde J(X,Y)$ in the $L^\infty$ distance, as discussed in the previous section.

Note, all integrals in $\mathcal M$ can be written as ``nice'' equations using $\theta, z, w, p, q, x,t$ and \eqref{xt_co}. For example,
\[
2\iint H(x-y,t-s)\theta_s(y,s)dyds=
\iint \tilde H\,\frac{pq}{4c}(\sin{w}\cos^2\frac{z}{2}-\sin{z}\cos^2\frac{w}{2})\,dXdY,
\]
where 
$$
\tilde H=H(x-y(X,Y),t-s(X,Y)).
$$
So the continuity of $\mathcal M$ in super norm can be proved by the  Lipschitz continuous dependence of $(\theta, z, w, p, q, x,t)(X,Y)$  on $J$ in the $L^\infty$ distance.

Clearly, the $L^\infty$ and $L^*$ norm are different in the whole domain, but they are the same in any bounded domain.
Thanks to the finite propagation property for \eqref{192}, we can split the region into $\Omega^0$ and the rest, where for any $J(x,t)\in \mathcal K$, singularity only happens on $\Omega^0$. The continuity of $\mathcal M(J)$ on $J$ is clear for smooth solutions on $\Omega/\Omega^0$. One can use both $L^2$ and $L^\infty$ norms estimate to cover the $L^*$ space.
In $\Omega^0$, we know that 
\[
\|J_2-J_1\|_{L^\infty}=\lim_{p\rightarrow \infty}\|J_2-J_1\|_{L^p}\leq \|J_2-J_1\|_{L^*}.
\]
Hence, we proved (b).
 
\bigskip
Therefore, by (a) and (b), an application of the Schauder Fixed-Point Theorem gives a fixed point $J^*(x,t)\in \mathcal K$ of $\mathcal T$, that is,
\beq\label{fix}
\mathcal M(J^*)(x,t)=J^*(x,t).
\eeq


\bigskip
Finally, let's fix $J=J^*$ in \eqref{fix}. Then  for any $0\leq t\leq \delta$, by previous results, we know that,
 \beq\label{weak0}
\theta_t(\cdot,t),\ \theta_x(\cdot,t)\in L^2(\mathbb R)\;\mbox{ and }\;  u,\, J\in L^\infty\cap L^2(\mathbb R\times [0,\delta]),\qquad
\eeq
and by applying standard heat equation theory to \eqref{key1},
\[
u\in L^2([0,\delta],H^1(\R)),\quad
u_t\in L^2([0,\delta],H^{-1}(\R)).
\]
Furthermore, from
\[
v=\int_{-\infty}^x u\,dz,
\]
one has 
\[v_t-v_{xx}=\theta_t,\]
in the $L^2$ sense by \eqref{key1}, which means for any test function $\phi\in C_0^\infty(\mathbb R\times [0,\delta])$,
\beq\label{weak 4}
\iint v_t\phi\,dx dt= \iint (v_{xx}+\theta_t)\phi\,dx dt
=\iint (u_x+\theta_t)\phi\,dx dt,
\eeq
where this and also the following integrals are all on $\mathbb R\times [0,\delta]$.
By \eqref{weak2}, we also know that
\beq\label{weak 3}
\iint \theta_t \phi_t dxdt=\iint \left((c(\theta)\phi)_x c(\theta)\theta_x +\theta_t \phi+J\phi\right)\,dxdt.
\eeq

Now we are ready to show that 
\[
J=v_t, \quad a.e..
\]
First notice that for any test function $\phi\in C_0^\infty(\mathbb R\times [0,\delta])$, there exist a function $\psi\in C_0^\infty(\mathbb R\times [0,\delta])$ such that 
\beq\label{psiphi}
\psi_t+\psi_{xx}+\psi=\phi.
\eeq
In fact, $\psi$ can be found by first solving (\ref{psiphi}) with an initial boundary value problem with zero boundary conditions on some bounded interval containing the support of $\phi$, then doing an zero extension.
By \eqref{weak 4} and \eqref{weak 3}, 
\[
\iint v_t\psi_t\,dx dt= \iint (v_{xx}+\theta_t)\psi_t\,dx dt
\]
and
\[
\iint \theta_t \psi_t dxdt=\iint \left((c(\theta)\psi)_x c(\theta)\theta_x +\theta_t \psi+J\psi\right)\,dxdt.
\]
Add above two equations up to get
 \[
\iint v_t(\psi_t+\psi_{xx})\,dx dt=\iint \left((c(\theta)\psi)_x c(\theta)\theta_x +\theta_t \psi+J\psi\right)\,dxdt.
\]
Furthermore, by \eqref{key2} we mean that
\[
\iint J(\psi_t+\psi_{xx})\,dxdt=\iint \left((c(\theta)\psi)_x c(\theta)\theta_x +u_x\psi+2\theta_t \psi\right)\,dxdt.
\]
Now comparing the above two equations, and using \eqref{weak 4} and \eqref{psiphi}, we have
\beq\label{weak 8}
\iint (J-v_t)\phi\, dxdt=0,
\eeq
which shows that $J=v_t$, a.e., and 
\beq\label{weak 6}
\iint J\phi\,dxdt =\iint (u_x+\theta_t)\phi\,dxdt,
\eeq
The equation \eqref{weak1} is satisfied due to \eqref{weak 3} and \eqref{weak 8}. 

Finally, it follows from  \eqref{weak 6}, \eqref{weak0} and \eqref{6.20}  that 
\[
 u_x(x,t)\in L^{\infty}([0,\delta], L^2_{loc}(\mathbb R)).
\]
The H\"older continuous properties of $u$ and $\theta$ in Theorem \ref{thmplc1} can be easily shown by the Sobolev embedding from $H^1_{loc}$ to $C^{1/2}$ in one space dimension.


\subsection{Energy Estimate\label{sec7}}
We now prove the energy estimate for the weak solution established in Section \ref{sec_6.2}, which provides the energy estimate in  Theorem \ref{thmplc1} and allows an extension of the solution to $[0,T]$.

\begin{thm}\label{lemma3.2} Fix $T>0$. For any weak solution of system \eqref{simeqn0} constructed in Section \ref{sec_6.2}, one has, for $t\in[0,T]$, 
\beq\label{engineq2}
\mathcal{E}(t)\leq \mathcal{E}(0) -\iint_{\R\times [0,t]}(v_t^2+\theta^2_t)\,dxdt.
\eeq
\end{thm}
	
\begin{proof} We first consider the bounded domain $D_{t}$ in the $(X,Y)$-plane in Figure \ref{fig:fig3}, and using the corresponding notation \eqref{last1}-\eqref{last2}. Recall that $\mathcal{D}$ is the region in the $(x,t)$-plane transformed from $D_t$.

We start our proof from inequality \eqref{eqn3.3}, where $J$ in the last integral can be replaced by $v_t$, that is,
\beq\label{eEst}
\begin{split}
\int_a^b (\theta_t^2+c^2(\theta)\theta_x^2)(x,t)\,dx \leq& \int_d^c (\theta_t^2+c^2(\theta)\theta_x^2)(x,0)\,dx\\
&-2\iint_{\mathcal{D}}\theta^2_t\,dxdt-2\iint_{\mathcal{D}}v_t\theta_t\,dxdt.
\end{split}
\eeq

By \eqref{weak 6} and the discussion in Section \ref{sec_6.2}, we know 
$
J=v_t=v_{xx}+\theta_t
$
holds true in the $L^2(\mathbb R\times[0,T])$ sense. Thus,
\beq\label{eqn3.4}
\begin{split}
\iint_{\mathcal{D}}v_t\theta_t\,dxdt=\iint_{\mathcal{D}}(v_t^2-v_{xx} v_t)\,dxdt,
\end{split}
\eeq
where $v_{xx}$, $\theta_t$, $v_t\in L^2(\mathcal{D})$.
Integrating by parts, the second term becomes
\beq\label{eqn3.5}
\begin{split}
-\iint_{\mathcal{D}}v_{xx}v_t\,dxdt
=&\iint_{\mathcal{D}}v_{x}v_{xt}\,dxdt+\int_{AD}\frac{v_xv_t}{\sqrt{1+c^2}}\,ds-\int_{CB}\frac{v_xv_t}{\sqrt{1+c^2}}\,ds\\
=&\frac12\int_a^b|v_x|^2(x,t)\,dx-\frac12\int_d^c|v_x|^2(x,0)\,dx\\
&-\int_{0}^t(v_tv_x)(x^+(t),t)\,dt-\int_{0}^t (v_tv_x)(x^-(t),t)\,dt
\end{split}
\eeq
where $x^+(t)$ and $x^-(t)$ are characteristic $DA$ and $CB$ respectively.
Substitute this identity into \eqref{eqn3.4} to get 
\beq\label{eqn3.6}
\begin{split}
\iint_{\mathcal{D}}v_t\theta_t\,dxdt=&\iint_{\mathcal{D}}v_t^2\,dxdt +\frac12\int_a^b u^2(x,t)\,dx-\frac12\int_d^c u^2(x,0)\,dx\\
&-\int_{0}^t(v_tv_x)(x^+(t),t)\,dt-\int_{0}^t (v_tv_x)(x^-(t),t)\,dt.
\end{split}
\eeq
Because $v_t$ is uniformly bounded and $v_x=u\in L^2$, 
 \[-\int_{0}^t(v_tv_x)(x^+(t),t)\,dt-\int_{0}^t (v_tv_x)(x^-(t),t)\,dt\to 0\;\mbox{ as }\; (a,b)\rightarrow (-\infty, \infty).\]
 Taking $(a,b)\to (-\infty,\infty)$  (so $(d,c)\to (-\infty,\infty)$ too) in \eqref{eEst} and \eqref{eqn3.6}, we have
\beq\label{eqn3.3-2}
\begin{split}
\mathcal E(t)
\leq \mathcal E(0) -\iint_{\R\times [0,t]}(v_t^2+\theta^2_t)\,dxdt.
\end{split}
\eeq
This completes the proof.
\end{proof}

By Lemma \ref{lemma3.1}, we know that the time step $\delta$ only depends on the bound of $\mathcal{E}(t)$ in \eqref{6.20}. The previous theorem gives an a priori bound on  $\mathcal{E}(t)$, which only depends on the initial condition. By this piece of information, we know $\delta$ is uniformly positive, so one can obtain a solution on $[0,T]$ in finite many time steps. Alternatively, one can choose $K$ larger than the a priori bound on  $\|J(x,t)-J^0(x,t)\|_{C^{\alpha}\cap L^\infty\cap L^2(\Omega_T)}$ only depending on the initial data, then directly find the fixed point for $t\in[0,T]$. Here there seems to be repetition in our presentation from local to global existence. We feel this might be helpful for the readers to understand the ideas for this involved proof. This completes the proof of Theorem \ref{thmplc1}.

\begin{remark}
If the solution has no energy concentration at time $t_1$, i.e. $\cos w$ and $\cos z$ are not $-1$ or equivalently $\theta_t$
and $\theta_x$ both have no blowup at $t_1$, then for any $t_2\geq t_1$, 
\beq\label{energydecay}\mathcal E(t_2)\leq \mathcal E(t_1)-\int_{t_1}^{t_2}\int_{\R}(v_t^2+\theta^2_t)\,dxdt.
\eeq
In fact, one can still prove it using the same method in the last theorem.

However, if  solution has energy concentration at time $t_1$, \eqref{energydecay} might not be true, because some energy might be later released from concentration. In this case, one cannot get the last step of
\eqref{eqn3.3}, where
\[
\begin{split}
\int_{DC}\frac{1-\cos w}{4}p&\,dX-\frac{1-\cos z}{4}q\,dY\\
>&\int_{DC\cap(\cos w\neq -1)}\frac{1-\cos w}{4}p\,dX+ 
\int_{CD\cap(\cos z\neq -1)}\frac{1-\cos z}{4}q\,dY\\
=&\int_d^c \frac{1}{2}(\theta_t^2+c^2(\theta)\theta_x^2)(x,t_1)\,dx,
\end{split}
\]
with $DC$ on the curve $t=t_1$.
\end{remark}

\medskip
\appendix

\section{}

\subsection{Derivation of system \eqref{time_poi}\label{Sec_A1}}

We consider the following form of solutions to the system \eqref{wlce}
$$
\bu(x,t)=(0,0,u(x,t))^T\;\mbox{ and }\; 
\n(x,t)=\big(\sin\theta(x,t), 0, \cos\theta(x,t)\big)^T.
$$
It is easy to see that $\nabla\cdot \bu(x,t)=0 $, and 
$\bu\cdot \nabla \bu=\bu\cdot\nabla \n =\bu\cdot\nabla \dot{\n}=0$.
Direct computation implies
\beq\notag
D=\frac{1}{2}(\nabla \bu+\nabla^T\bu)=\frac{1}{2}
\left[
\begin{array}{ccc}
0&0&u_x\\
0&0&0\\
u_x&0&0
\end{array}
\right],
\eeq
\beq\notag
\omega=\frac{1}{2}(\nabla \bu-\nabla^T\bu)=\frac{1}{2}
\left[
\begin{array}{ccc}
0&0&-u_x\\
0&0&0\\
u_x&0&0
\end{array}
\right],
\eeq
\beq\notag
\n\otimes\n=
\left[
\begin{array}{ccc}
\sin^2\theta&0&\sin\theta\cos\theta\\
0&0&0\\
\sin\theta\cos\theta&0&\cos^2\theta
\end{array}
\right].
\eeq
Then 
\beq\notag
N=\dot{\n}-\omega \n=\left(\theta_t+\frac12u_x\right)\big(\cos\theta, 0,- \sin\theta\big)^T,
\eeq
\beq\notag
N\otimes\n=\left(\theta_t+\frac12u_x\right)
\left[
\begin{array}{ccc}
\sin\theta\cos\theta&0&\cos^2\theta\\
0&0&0\\
-\sin^2\theta&0&-\sin\theta\cos\theta
\end{array}
\right],
\eeq
\beq\notag
\n\otimes N=\left(\theta_t+\frac12u_x\right)
\left[
\begin{array}{ccc}
\sin\theta\cos\theta&0&-\sin^2\theta\\
0&0&0\\
\cos^2\theta&0&-\sin\theta\cos\theta
\end{array}
\right].
\eeq
And also
\beq\notag
D\n=\frac12u_x\big(\cos\theta, 0,\sin\theta\big)^T,
\eeq
\beq\notag
\n^TD\n=u_x\cos\theta\sin\theta,
\eeq
\beq\notag
D\n\otimes \n=\frac12u_x
\left[
\begin{array}{ccc}
\sin\theta\cos\theta&0&\cos^2\theta\\
0&0&0\\
\sin^2\theta&0&\sin\theta\cos\theta
\end{array}
\right],
\eeq
\beq\notag
\n\otimes D\n=\frac12u_x
\left[
\begin{array}{ccc}
\sin\theta\cos\theta&0&\sin^2\theta\\
0&0&0\\
\cos^2\theta&0&\sin\theta\cos\theta
\end{array}
\right].
\eeq
Hence
\beq\notag
\begin{split}
{\bf g}=&\gamma_1N+\gamma_2D\n\\
=&\gamma_1\theta_t\big(\cos\theta, 0,- \sin\theta\big)^T+\frac12u_x\big((\gamma_1+\gamma_2)\cos\theta, 0,(\gamma_2-\gamma_1) \sin\theta\big)^T.
\end{split}
\eeq
Since the last term in Oseen-Frank energy density is null Lagrangian term, without loss of generalization, we only compute the first three terms. The Oseen-Frank energy density of this case will be 
\beq\notag
W(\n,\nabla\n)=K_1(\n^1_x)^2+K_3\big((\n^1\n^1_x)^2+(\n^1\n^3_x)^2\big)=\frac12\theta^2_x\big(K_1\cos^2\theta+K_3\sin^2\theta\big),
\eeq
where $\n^i$ is the $i$-th component of $\n$.
Thus
\beq\notag
\frac{\partial W}{\partial\n}=\big(K_3\theta_x^2\sin\theta, 0,0\big)^T,
\eeq
\beq\notag
\frac{\partial W}{\partial\nabla\n}=
\left[
\begin{array}{ccc}
K_1\theta_x\cos\theta+K_3\theta_x\sin^2\theta\cos\theta&0&0\\
0&0&0\\
-K_3\theta_x\sin^3\theta&0&0
\end{array}
\right].
\eeq
The Lagrangian constant 
\beq\notag
\begin{split}
\gamma=&\frac{\partial W}{\partial\n}\cdot\n+\gamma_2D\n\cdot\n
-\nabla\cdot\left(\frac{\partial W}{\partial\nabla\n}\right)\cdot\n-|\n_t|^2\\
=&(K_1+2K_3)\theta_x^2\sin^2\theta-K_1\theta_{xx}\sin\theta\cos\theta+\gamma_2u_x\sin\theta\cos\theta-|\theta_t|^2.
\end{split}
\eeq
We are ready to derive the system \eqref{time_poi}. We first work on the equation of $\theta$. By the third equation of \eqref{wlce}, we have
\beq\notag
\begin{split}
\n_{tt}=&\theta_{tt}\big(\cos\theta,0,-\sin\theta\big)^T+|\theta_t|^2\big(-\sin\theta,0,-\cos\theta\big)^T\\
=&\gamma\n-\frac{\partial W}{\partial\n}-{\bf g}+\nabla\cdot\left(\frac{\partial W}{\partial\nabla\n}\right)\\
=&-|\theta_t|^2\big(\sin\theta,0,\cos\theta\big)^T-\gamma_1\theta_t\big(\cos\theta, 0,- \sin\theta\big)^T+u_x{\mathbf T_1}+K_1\mathbf T_2+K_3\mathbf T_3.
\end{split}
\eeq
Then
\beq\label{eqnAA1}
\begin{split}
(\theta_{tt}+\gamma_1\theta_t)\big(\cos\theta,0,-\sin\theta\big)^T
=u_x{\mathbf T_1}+K_1\mathbf T_2+K_3\mathbf T_3.
\end{split}
\eeq
Here the vector $\mathbf T_1$ is given by 
\beq\label{eqnT1}
\begin{split}
\mathbf T_1=&\gamma_2\sin\theta\cos\theta\big(\sin\theta, 0,\cos\theta\big)^T-\frac12\big((\gamma_1+\gamma_2)\cos\theta, 0,(\gamma_2-\gamma_1) \sin\theta\big)^T\\
=&-\frac{\gamma_1}{2}\big(\cos\theta, 0,-\sin\theta\big)^T+\frac{\gamma_2}{2}\big(2\sin^2\theta\cos\theta-\cos\theta, 0,2\sin\theta\cos^2\theta-\sin\theta\big)^T\\
=&\left(-\frac{\gamma_1}{2}-\frac{\gamma_2}{2}\cos(2\theta)\right)\big(\cos\theta, 0,-\sin\theta\big)^T.
\end{split}
\eeq
The nonzero components of vector $\mathbf T_2$ is given by 
\begin{align*}
\mathbf T_2^1=&\theta_x^2\sin^3\theta-\theta_{xx}\sin^2\theta\cos\theta+(\theta_x\cos\theta)_x\\
=&-\sin^2\theta(\theta_x\cos\theta)_x+(\theta_x\cos\theta)_x
=\cos^2\theta(\theta_x\cos\theta)_x,\\
\mathbf T_2^3=&\theta_x^2\sin^2\theta\cos\theta-\theta_{xx}\sin\theta\cos^2\theta
=-\sin\theta\cos\theta(\theta_x\cos\theta)_x.
 \end{align*}
So the vector $\mathbf T_2$ is 
\beq\label{eqnT2}
\begin{split}
\mathbf T_2=
\cos\theta(\theta_x\cos\theta)_x\big(\cos\theta, 0,-\sin\theta\big)^T.
\end{split}
\eeq
Similarly, the vector $\mathbf T_3$ is 
\beq\label{eqnT3}
\begin{split}
\mathbf T_3=
\sin\theta(\theta_x\sin\theta)_x\big(\cos\theta, 0,-\sin\theta\big)^T.
\end{split}
\eeq
Plugging \eqref{eqnT1}, \eqref{eqnT2} and \eqref{eqnT3} into \eqref{eqnAA1}, we obtain
$$
\theta_{tt}+\gamma_1\theta_t=K_1\cos\theta(\theta_x\cos\theta)_x+K_3\sin\theta(\theta_x\sin\theta)_x+\left(-\frac{\gamma_1}{2}-\frac{\gamma_2}{2}\cos(2\theta)\right)u_x,
$$
which is exactly the second equation in \eqref{time_poi}. 

For the equation of $u$, direct computation gives 
\beq\notag
\nabla\cdot\left(\frac{\partial W}{\partial\nabla\n}\otimes\nabla\n\right)
=\big((K_1\theta_x^2\cos^2\theta)_x+(K_3\theta_x^2\sin^2\theta)_x, 0,0\big)^T
\eeq
and 
\beq\notag
\begin{split}
\nabla\cdot\sigma
=&\alpha_1\big((u_x\sin^3\theta\cos\theta)_x,0,(u_x\sin^2\theta\cos^2\theta)_x\big)^T\\
&+\alpha_2\left(\big((\theta_t+\frac12u_x)\sin\theta\cos\theta\big)_x,0,-\big((\theta_t+\frac12u_x)\sin^2\theta\big)_x\right)^T\\
&+\alpha_3\left(\big((\theta_t+\frac12u_x)\sin\theta\cos\theta\big)_x,0,\big((\theta_t+\frac12u_x)\cos^2\theta\big)_x\right)^T\\
&+\frac{\alpha_4}{2}\big(0,0, u_{xx}\big)^T
+\frac{\alpha_5}{2}\big((u_x\sin\theta\cos\theta)_x,0,(u_x\sin^2\theta)_x\big)^T\\
&+\frac{\alpha_6}{2}\big((u_x\sin\theta\cos\theta)_x,0,(u_x\cos^2\theta)_x\big)^T.
\end{split}
\eeq
Therefore the first equation of system \eqref{wlce} can be written into following three equations
\beq\label{eqnPx}
\begin{split}
P_x=&K_1(\theta_x^2\cos^2\theta)_x+K_3(\theta_x^2\sin^2\theta)_x+\alpha_1(u_x\sin^3\theta\cos\theta)_x\\
&+\frac12\left(\alpha_2+\alpha_3+\alpha_5+\alpha_6\right)(u_x\sin\theta\cos\theta)_x+(\alpha_2+\alpha_3)\big(\theta_t\sin\theta\cos\theta\big)_x,
\end{split}
\eeq
\beq\label{eqnPy}
P_y=0,
\eeq
\beq\label{eqnPz}
\begin{split}
u_t+P_z=&\frac{\alpha_4}{2}u_{xx}
+\alpha_1(u_x\sin^2\theta\cos^2\theta)_x
-\alpha_2\big((\theta_t+\frac12u_x)\sin^2\theta\big)_x\\
&+\alpha_3\big((\theta_t+\frac12u_x)\cos^2\theta\big)_x
+\frac{\alpha_5}{2}(u_x\sin^2\theta)_x
+\frac{\alpha_6}{2}(u_x\cos^2\theta)_x.
\end{split}
\eeq
By these equations, one can obtain that $P_z=a$ for some constant $a$. The right hand side of \eqref{eqnPz} can be rewritten as $(g(\theta)u_x+h(\theta)\theta_t)_x$ 
where $g(\theta)$ and $h(\theta)$ is defined as \eqref{fgh}. 
Therefore, we obtain the first equation of \eqref{time_poi}.


\bigskip
\subsection{Derivation of system (\ref{semilinear_new})}\label{Sec_A3}
We will in fact derive the semilinear system in $XY$-coordinates for \eqref{time_poi} with $\nu=\rho=1$ and $a=0$. Recall, from (\ref{defwz}) and (\ref{defpq}), that we have introduced
\[w=2\arctan R,\quad z=2\arctan S, \quad p=\frac{1+R^2}{X_x},\quad q=-\frac{1+S^2}{Y_x}.\]
It is easy to have that 
\[
R=\tan\frac{w}{2},\quad \frac{R}{1+R^2}=\frac{1}{2}\sin w, \quad \frac{1}{1+R^2}=\cos^2\frac{w}{2},
  \]
\[S=\tan\frac{z}{2},
\quad
\frac{S}{1+S^2}=\frac{1}{2}\sin z,
 \quad
\frac{1}{1+S^2}=\cos^2\frac{z}{2}.\]

By (\ref{trans}) and (\ref{RxSx}), we have
\[
\theta_X=\frac{\sin w}{4c}p,\quad \theta_Y=\frac{\sin z}{4c}q.
\]

Denote 
\[
U=h(\theta)u_x.
\]

By \eqref{RxSx}, we have
\beq\label{A8}
\left\{
\begin{array}{rcl}
 (S^2)_t+c(\theta)(S^2)_x &=& \frac{c'}{2c}(S^3-R^2S)-\gamma_1(RS+S^2)-2SU,\\
 (R^2)_t-c(\theta)(R^2)_x &=& \frac{c'}{2c} (R^3-RS^2)-\gamma_1(RS+R^2)-2RU,
\end{array}\right.
\eeq
so
\beq
\left\{
\begin{array}{rcl}
 (S^2)_X &=& \frac{1}{2c}p\frac{1}{1+R^2}\Big\{\frac{c'}{2c}(S^3-R^2S)-\gamma_1(RS+S^2)-2SU\Big\},\\
 (R^2)_Y &=& \frac{1}{2c}q\frac{1}{1+S^2}\Big\{\frac{c'}{2c} (R^3-RS^2)-\gamma_1(RS+R^2)-2RU\Big\}.
\end{array}\right.
\eeq
Hence,
\beq
\begin{split}
z_X=&\textstyle\frac{1}{1+S^2}\frac{1}{S}(S^2)_X
\nn\\
	     =&\textstyle
\frac{1}{2c}p\frac{1}{(1+R^2)(1+S^2)}\Big\{\frac{c'}{2c}(S^2-R^2)-\gamma_1(R+S)-2U\Big\}    \nn	
\\=&\textstyle p
\Big\{\frac{c'}{4c^2}(\cos^2\frac{w}{2}-\cos^2\frac{z}{2})-\frac{\gamma_1}{4c}(\sin w\cos^2\frac{z}{2}+\sin z\cos^2\frac{w}{2})-\frac{1}{c}\cos^2\frac{z}{2}\cos^2\frac{w}{2}U\Big\} .  \nn
\end{split}
\eeq
Similarly,
\[\textstyle
w_Y=q
\Big\{\frac{c'}{4c^2}(\cos^2\frac{z}{2}-\cos^2\frac{w}{2})-\frac{\gamma_1}{4c}(\sin w\cos^2\frac{z}{2}+\sin z\cos^2\frac{w}{2})-\frac{1}{c}\cos^2\frac{z}{2}\cos^2\frac{w}{2}U\Big\}  .
\]


\medskip

On the other hand, using $X_t-cX_x=0$, we have
\[
X_{tx}-cX_{xx}=\frac{c'}{2c}(R-S)X_x.
\]
Then
\begin{align*}
p_Y =&\textstyle\frac{1}{X_x}(R^2)_Y-\frac{q\cos^2{\frac{w}{2}}}{2c}\frac{1+R^2}{X_x^2}(X_{xt}-cX_{xx}) \\
=&\textstyle\frac{pq}{2c}\frac{1}{(1+R^2)(1+S^2)}\Big\{\frac{c'}{2c} (R^3-RS^2)-\gamma_1(RS+R^2)-2RU\Big\}
-\frac{pq}{2c}\frac{c'}{2c}\frac{R-S}{1+S^2}\nn\\
=&\textstyle\frac{pq}{2c}\frac{1}{(1+R^2)(1+S^2)}\Big\{\frac{c'}{2c} (-R-RS^2)-\gamma_1(RS+R^2)-2RU\Big\}
+\frac{pq}{2c}\frac{c'}{2c}\frac{S}{1+S^2}\nn\\
=&\textstyle pq\frac{c'}{8c^2}(\sin z-\sin w)
-\gamma_1\frac{pq}{2c}[\frac{1}{4}\sin w\sin z+\sin^2 \frac{w}{2}\cos^2\frac{z}{2}]
-\frac{pq}{2c}U\sin w\cos^2\frac{z}{2}.
\end{align*}
Similarly, we have 
\[
q_X=\textstyle pq\frac{c'}{8c^2}(\sin w-\sin z)
-\gamma_1\frac{pq}{2c}[\frac{1}{4}\sin w\sin z+\sin^2 \frac{z}{2}\cos^2\frac{w}{2}]
-\frac{pq}{2c}U\sin z\cos^2\frac{w}{2}.
\]

In summary, we have the following system of equations.
\beq\label{semilinear}
\begin{split}
\theta_X&\textstyle=\frac{\sin w}{4c}p,\qquad \theta_Y=\frac{\sin z}{4c}q\\
z_X&\textstyle=p
\Big\{\frac{c'}{4c^2}(\cos^2\frac{w}{2}-\cos^2\frac{z}{2})-\frac{\gamma_1}{4c}(\sin w\cos^2\frac{z}{2}+\sin z\cos^2\frac{w}{2})-\frac{1}{c}\cos^2\frac{z}{2}\cos^2\frac{w}{2}U\Big\}   \nn\\
w_Y&\textstyle=q
\Big\{\frac{c'}{4c^2}(\cos^2\frac{z}{2}-\cos^2\frac{w}{2})-\frac{\gamma_1}{4c}(\sin w\cos^2\frac{z}{2}+\sin z\cos^2\frac{w}{2})-\frac{1}{c}\cos^2\frac{z}{2}\cos^2\frac{w}{2}U\Big\}  
\nn\\
p_Y&\textstyle=pq\frac{c'}{8c^2}(\sin z-\sin w)
-\gamma_1\frac{pq}{2c}[\frac{1}{4}\sin w\sin z+\sin^2 \frac{w}{2}\cos^2\frac{z}{2}]
-\frac{pq}{2c}U\sin w\cos^2\frac{z}{2}\nn\\
q_X&\textstyle=pq\frac{c'}{8c^2}(\sin w-\sin z)
-\gamma_1\frac{pq}{2c}[\frac{1}{4}\sin w\sin z+\sin^2 \frac{z}{2}\cos^2\frac{w}{2}]
-\frac{pq}{2c}U\sin z\cos^2\frac{w}{2}\nn
\end{split}
\eeq
where recall that $U=h(\theta)u_x.$ 

Denote $J=g(\theta)u_x+h(\theta)\theta_t$.
Then
\[
U=\frac{h(\theta)}{g(\theta)}(J-h(\theta)\theta_t).
\]
In terms of the variable $J$, the above system is 
\beq\label{semilinearA}
\begin{split}
&\theta_X=\frac{\sin w}{4c}p,\quad \theta_Y=\frac{\sin z}{4c}q\\
&\textstyle z_X=p
\Big\{\frac{c'}{4c^2}(\cos^2\frac{w}{2}-\cos^2\frac{z}{2})+\frac{\frac{h^2(\theta)}{g(\theta)}-\gamma_1}{4c}(\sin w\cos^2\frac{z}{2}+\sin z\cos^2\frac{w}{2})-\frac{h(\theta)}{cg(\theta)}J\cos^2\frac{z}{2}\cos^2\frac{w}{2} \Big\} \\
&\textstyle w_Y=q
\Big\{\frac{c'}{4c^2}(\cos^2\frac{z}{2}-\cos^2\frac{w}{2})+\frac{\frac{h^2(\theta)}{g(\theta)}-\gamma_1}{4c}(\sin w\cos^2\frac{z}{2}+\sin z\cos^2\frac{w}{2})-\frac{h(\theta)}{cg(\theta)}J\cos^2\frac{z}{2}\cos^2\frac{w}{2}\Big\} \\
&\textstyle p_Y=pq\Big\{\frac{c'}{8c^2}(\sin z-\sin w)
+\frac{\frac{h^2(\theta)}{g(\theta)}-\gamma_1}{2c}(\frac{1}{4}\sin w\sin z+\sin^2 \frac{w}{2}\cos^2\frac{z}{2})
-\frac{h(\theta)}{2cg(\theta)}J\sin w\cos^2\frac{z}{2}\Big\}\\
&\textstyle q_X=pq\Big\{\frac{c'}{8c^2}(\sin w-\sin z)
+\frac{\frac{h^2(\theta)}{g(\theta)}-\gamma_1}{2c}(\frac{1}{4}\sin w\sin z+\sin^2 \frac{z}{2}\cos^2\frac{w}{2})
-\frac{h(\theta)}{2cg(\theta)}J\sin z\cos^2\frac{w}{2}\Big\}
\end{split}
\eeq
where the coefficient $\frac{h^2(\theta)}{g(\theta)}-\gamma_1<0$ by \eqref{damping}.

For the special case where 
$\gamma_1=2$, $\gamma_2=0$ and $g(\theta)=h(\theta)=1$,
system (\ref{semilinearA}) reduces to system \eqref{semilinear_new}.

\subsection{H\"older continuity of $\mathcal M(J)$ in \S \ref{sec_6.2}}
\label{secA3}

First, we consider three types of integrals in the definition of $J$ in \eqref{hatvtL}, 
\beq\notag
L_i(x,t):=\int_{\R}H(x-y,t)f(y)\,dy
\eeq
\beq\notag
L_{ii}(x,t):=\int_0^t\int_{\mathbb R}H(x-y,t-s)g(y,s)\,dyds
\eeq
\beq\notag
L_{iii}(x,t):=\int_0^t\int_{\mathbb R}H_x(x-y,t-s)f(y,s)\,dyds.
\eeq

We first prove the following lemma working generally.

\begin{lemma}\label{lem2}
If $f(x,t)\in L^{\infty}([0,T],  L^2(\R))$, $g(x,t)\in L^{\infty}([0,T],  L^1(\R))$ or $g(x,t)\in L^{\infty}([0,T],  L^2(\R))$ we have $L_j(x,t)$ is H\"older continuous w.r.t $x$ and $t$, for all $j=i,ii,iii$, with  exponent $\beta\in(0,\frac{1}{4})$.
\end{lemma}
 \begin{proof}
We only provide the proof of H\"older continuity w.r.t $x$ for $L_{iii}$, the rest cases can be proved similarly. For any $x_1<x_2$, it is sufficient to show
\beq\label{hbdx}
\left|\int_0^t\int_{\R} \frac{H_y(x_2-y,t-s)-H_y(x_1-y,t-s)}{(x_2-x_1)^\beta}f(y,s) dy ds\right|\leq\hbox{Constant}\, \cdot\, t^{\frac{1}{4}-\frac{\beta}{2}},
\eeq
for some $\beta\in(0,1)$.
Since
$$
H_x(x,t)=-\frac{1}{4\sqrt{\pi}}\frac{x}{t^{3/2}}\exp\left(-\frac{x^2}{4t}\right),
$$
\beq\notag
\begin{split}
&H_x(x_2-y,t-s)-H_x(x_1-y,t-s)\\
&\quad =-\frac{1}{4\sqrt{\pi}}\frac{1}{(t-s)^{3/2}}
\left\{
(x_2-x_1)\exp\left(-\frac{(x_2-y)^2}{4(t-s)}\right)\right.\\
&\qquad -\left.(x_1-y)\left[\exp\left(-\frac{(x_2-y)^2}{4(t-s)}\right)-\exp\left(-\frac{(x_1-y)^2}{4(t-s)}\right)\right]
\right\}\\
&\quad :=I_1+I_2.
\end{split}
\eeq
The integral related to the first term $I_1$ can be estimated as follows
\beq\label{clm1}
\begin{split}
&\int_0^t\int_{\R}\frac{1}{(t-s)^{3/2}}(x_2-x_1)^{1-\beta}\exp\left(-\frac{(x_2-y)^2}{4(t-s)}\right)|f(y,s)|\,dyds\\
\leq& C\int_0^t\int_{\R}\frac{1}{(t-s)^{3/2}}\left(|x_2-y|^{1-\beta}+|x_1-y|^{1-d}\right)\exp\left(-\frac{(x_2-y)^2}{4(t-s)}\right)|f(y,s)|\,dyds.
\end{split}
\eeq
By the same argument in \eqref{key}, and choosing $0<\beta<\frac{1}{8}$, the first term in \eqref{clm1} should be controlled by 
\beq\label{clm2}
\int_0^t\int_{\R}\frac{1}{(t-s)^{3/2}}|x_2-y|^{1-\beta}\exp\left(-\frac{(x_2-y)^2}{4(t-s)}\right)|f(y,s)|\,dyds\leq Ct^{\frac14-\frac \beta 2}\|f\|_{L^{\infty}((0,T), L^2(\R))}.
\eeq
For the second term in \eqref{clm1}
\beq\label{clm3}
\begin{split}
&\int_0^t\int_{\R}\frac{1}{(t-s)^{3/2}}|x_1-y|^{1-\beta}\exp\left(-\frac{(x_2-y)^2}{4(t-s)}\right)|f(y,s)|\,dyds\\
\leq& C\left(\int_0^t\int_{\R}\frac{|x_1-y|^{2(1-\beta)}}{(t-s)^{9/4-\beta/2}}\exp\left(-\frac{(x_2-y)^2}{2(t-s)}\right)\,dyds\right)^{\frac12}\left(\int_0^t\int_{\R}\frac{|f(y,s)|^2}{(t-s)^{3/4+\beta/2}}\,dyds\right)^{\frac12}\\
\leq& C\left(\int_0^t\int_{\R}\frac{|\sqrt{t-s}u|^{2(1-\beta)}}{(t-s)^{7/4-\beta/2}}\exp\left(-\frac{u^2}{2}-\frac{x_2-x_1}{\sqrt{t-s}}u\right)\,duds\right)^{\frac12}\left(\int_0^t\int_{\R}\frac{|f(y,s)|^2}{(t-s)^{3/4+\beta/2}}\,dyds\right)^{\frac12}\\
\leq& C\left(\int_0^t\int_{\R}\frac{|u|^{2(1-\beta)}}{(t-s)^{3/4+\beta/2}}\exp\left(-u\big(\frac{u}{2}+\frac{x_2-x_1}{\sqrt{t-s}}\big)\right)\,duds\right)^{\frac12}\left(\int_0^t\int_{\R}\frac{|f(y,s)|^2}{(t-s)^{3/4+\beta/2}}\,dyds\right)^{\frac12}\\
\leq&  Ct^{\frac{1}{4}-\frac \beta2}\|f\|_{L^{\infty}((0,T), L^2(\R))},
\end{split}
\eeq
where $x_1-y=u\sqrt{t-s}$, and 
\beq\notag
\begin{split}
&\exp\left(-\frac{(x_2-y)^2}{2(t-s)}\right)
=\exp\left(-\frac{(x_2-x_1+x_1-y)^2}{2(t-s)}\right)\\
=&\exp\left(-\frac{(x_2-x_1)^2}{2(t-s)}\right)\cdot\exp\left(-\frac{2(x_2-x_1)(x_1-y)}{2(t-s)}\right)\cdot\exp\left(-\frac{(x_1-y)^2}{2(t-s)}\right).
\end{split}
\eeq
Putting \eqref{clm2} and \eqref{clm3} into \eqref{clm1}, it holds
\beq\label{clm4}
\begin{split}
&\int_0^t\int_{\R}\frac{1}{(t-s)^{3/2}}(x_2-x_1)^{1-\beta}\exp\left(-\frac{(x_2-y)^2}{4(t-s)}\right)|f(y,s)|\,dyds\\
&\quad \leq Ct^{\frac{1}{4}-\frac \beta2}\|f\|_{L^{\infty}((0,T), L^2(\R))}.
\end{split}
\eeq

\medskip
On the other hand, by mean value theorem, there exists a $\xi\in(x_1,x_2)$ such that
\beq\notag
\begin{split}
e^{-\frac{(x_2-y)^2}{4(t-s)}}-e^{-\frac{(x_1-y)^2}{4(t-s)}}
=&-e^{-\frac{(\xi-y)^2}{4(t-s)}} \frac{\xi-y}{2(t-s)}(x_2-x_1)
\leq-e^{-\frac{(x_1-y)^2}{4(t-s)}}\frac{x_1-y}{2(t-s)}(x_2-x_1).
\end{split}
\eeq
Hence for the term related to $I_2$, it holds
\beq\label{clm5}
\begin{split}
&\int_0^t\int_{\R}\frac{1}{(t-s)^{3/2}}\frac{(x_1-y)^2}{t-s}(x_2-x_1)^{1-\beta}\exp\left(-\frac{(x_1-y)^2}{4(t-s)}\right) |f(y,s)|\,dyds\\
\leq& C\int_0^t\int_{\R}\frac{1}{(t-s)^{3/2}}\frac{(x_1-y)^2}{t-s}\left(|x_2-y|^{1-\beta}+|x_1-y|^{1-\beta}\right)\exp\left(-\frac{(x_1-y)^2}{4(t-s)}\right)|f(y,s)|\,dyds.
\end{split}
\eeq
The second term is similar to \eqref{key} and \eqref{clm2}
\beq\label{clm6}
\int_0^t\int_{\R}\frac{1}{(t-s)^{5/2}}|x_1-y|^{3-\beta}\exp\left(-\frac{(x_1-y)^2}{4(t-s)}\right)|f(y,s)|\,dyds\leq Ct^{\frac14-\frac \beta 2}\|f\|_{L^{\infty}((0,T), L^2(\R))}.
\eeq
For the first term, one has
\beq\label{clm7}
\begin{split}
& \int_0^t\int_{\R}\frac{1}{(t-s)^{3/2}}\frac{(x_1-y)^2}{t-s}|x_2-y|^{1-\beta}\exp\left(-\frac{(x_1-y)^2}{4(t-s)}\right)f(y,s)\,dyds\\
\leq& C\left(\int_0^t\int_{\R}\frac{|x_2-y|^{2(1-\beta)}}{(t-s)^{9/4-d/2}}\frac{(x_1-y)^4}{(t-s)^2}\exp\left(-\frac{(x_1-y)^2}{2(t-s)}\right)\,dyds\right)^{\frac12}\cdot\left(\int_0^t\int_{\R}\frac{|f(y,s)|^2}{(t-s)^{3/4+\beta/2}}\,dyds\right)^{\frac12}.
\end{split}
\eeq
Let $x_2-y=\sqrt{t-s}u$. Then it holds
\beq\label{clm8}
\begin{split}
&\int_0^t\int_{\R}\frac{|u|^{2(1-d)}}{(t-s)^{3/4+\beta/2}}\frac{(x_1-x_2+\sqrt{t-s}u)^4}{(t-s)^2}\exp\left(-\frac{(x_1-x_2+\sqrt{t-s}u)^2}{2(t-s)}\right)\,duds\\
\leq & C\int_0^t\int_{\R}\frac{|u|^{2(1-\beta)}}{(t-s)^{3/4+\beta/2}}\frac{(x_1-x_2)^4+(\sqrt{t-s}u)^4}{(t-s)^2}\exp\left(-\frac{(x_1-x_2+\sqrt{t-s}u)^2}{2(t-s)}\right)\,duds.
\end{split}
\eeq
It is easy to see 
\beq\notag
\begin{split}
&\exp\left(-\frac{(x_1-x_2+\sqrt{t-s}u)^2}{2(t-s)}\right)\\
&\quad =\exp\left(-\frac{(x_1-x_2)^2}{2(t-s)}\right)
\exp\left(-\frac{(x_1-x_2)u}{\sqrt{t-s}}\right)
\exp\left(-\frac{u^2}{2}\right)\\
&\quad \leq \exp\left(-\frac{u^2}{2}-\frac{(x_1-x_2)u}{\sqrt{t-s}}\right).
\end{split}
\eeq
Hence
\beq\label{clm9}
\begin{split}
 & \int_0^t\int_{\R}\frac{|u|^{2(1-\beta)}}{(t-s)^{3/4+\beta/2}}u^4\exp\left(-\frac{(x_1-x_2+\sqrt{t-s}u)^2}{2(t-s)}\right)\,duds\\
&\quad \leq \int_0^t\int_{\R}\frac{|u|^{2(1-\beta)}}{(t-s)^{3/4+\beta/2}}u^4 
 \exp\left(-\frac{u^2}{2}-\frac{(x_1-x_2)u}{\sqrt{t-s}}\right)\,duds\leq  Ct^{\frac14-\frac \beta 2},
\end{split}
\eeq
and 
\beq\label{clm10}
\begin{split}
&\int_0^t\int_{\R}\frac{|u|^{2(1-\beta)}}{(t-s)^{3/4+\beta/2}}\frac{(x_1-x_2)^4}{(t-s)^2}\exp\left(-\frac{(x_1-x_2+\sqrt{t-s}u)^2}{2(t-s)}\right)\,duds\\
=&\int_0^t\int_{\R}\frac{|u|^{2(1-\beta)}}{(t-s)^{3/4+\beta/2}}\frac{(x_1-x_2)^4}{(t-s)^2}\exp\left(-\frac{(x_1-x_2)^2}{2(t-s)}\right)
\exp\left(-\frac{u^2}{2}-\frac{(x_1-x_2)u}{\sqrt{t-s}}\right)\,duds\\
 \leq &\int_0^t\int_{\R}\frac{|u|^{2(1-\beta)}}{(t-s)^{3/4+\beta/2}}u^4 
 \exp\left(-\frac{u^2}{2}-\frac{(x_1-x_2)u}{\sqrt{t-s}}\right)\,duds\leq  Ct^{\frac14-\frac \beta 2}.
\end{split}
\eeq
Putting \eqref{clm9} and \eqref{clm10} into \eqref{clm8}, we obtain
\beq\label{clm11}
\begin{split}
&\int_0^t\int_{\R}\frac{|u|^{2(1-d)}}{(t-s)^{3/4+\beta/2}}\frac{(x_1-x_2+\sqrt{t-s}u)^4}{(t-s)^2}\exp\left(-\frac{(x_1-x_2+\sqrt{t-s}u)^2}{2(t-s)}\right)\,duds\\
&\quad \leq  Ct^{\frac14-\frac \beta 2}.
\end{split}
\eeq
Combining \eqref{clm11} with \eqref{clm6} and \eqref{clm7}, we have
\beq\label{clm12}
\begin{split}
&\int_0^t\int_{\R}\frac{1}{(t-s)^{3/2}}\frac{(x_1-y)^2}{t-s}(x_2-x_1)^{1-\beta}\exp\left(-\frac{(x_1-y)^2}{4(t-s)}\right) |f(y,s)|\,dyds\\
&\quad \leq Ct^{\frac{1}{4}-\frac \beta2}\|f\|_{L^{\infty}((0,T), L^2(\R))}.
\end{split}
\eeq
It is easy to see that \eqref{clm4} and \eqref{clm12} imply \eqref{hbdx}.

 Similarly, we can show the H\"older continuity of for $L_{iii}$ in $t$ , and
\beq\label{hbdx}
\left|\int_0^t\int_{\R} \frac{H_y(x_2-y,t-s)-H_y(x_1-y,t-s)}{(t_2-t_1)^\beta}f(y,s) dy ds\right|\leq C t^{\frac{1}{4}-\beta}.
\eeq

The proof of H\"older continuity for other terms are similar, and all bounds include a factor $t^{\frac{1}{4}-\beta}$.
\end{proof}

 \bigskip

\noindent
{\bf Acknowledgement:}  G. Chen's research is partially supported by NSF grant DMS-1715012, T. Huang's research is partially supported by  NSFC grant 11601333 and Shanghai NSF grant 16ZR1423800, and W. Liu's research is partially supported by Simons Foundation 
Mathematics and Physical Sciences-Collaboration Grants for Mathematicians \#581822.




\end{document}